\documentclass[11pt]{article}
\usepackage{latexsym,amsfonts,amssymb,amsmath,amsthm}
\usepackage{graphicx}

\usepackage[usenames,dvipsnames]{color}
\usepackage{ulem}

\begin{document}
\setlength{\baselineskip}{16pt}

\parindent 0.5cm
\evensidemargin 0cm \oddsidemargin 0cm \topmargin 0cm \textheight
22cm \textwidth 16cm \footskip 2cm \headsep 0cm

\newtheorem{theorem}{Theorem}[section]
\newtheorem{lemma}{Lemma}[section]
\newtheorem{proposition}{Proposition}[section]
\newtheorem{definition}{Definition}[section]
\newtheorem{example}{Example}[section]
\newtheorem{corollary}{Corollary}[section]

\newtheorem{remark}{Remark}[section]

\numberwithin{equation}{section}

\def\p{\partial}
\def\I{\textit}
\def\R{\mathbb R}
\def\C{\mathbb C}
\def\u{\underline}
\def\l{\lambda}
\def\a{\alpha}
\def\O{\Omega}
\def\e{\epsilon}
\def\ls{\lambda^*}
\def\D{\displaystyle}
\def\wyx{ \frac{w(y,t)}{w(x,t)}}
\def\imp{\Rightarrow}
\def\tE{\tilde E}
\def\tX{\tilde X}
\def\tH{\tilde H}
\def\tu{\tilde u}
\def\d{\mathcal D}
\def\aa{\mathcal A}
\def\DH{\mathcal D(\tH)}
\def\bE{\bar E}
\def\bH{\bar H}
\def\M{\mathcal M}
\renewcommand{\labelenumi}{(\arabic{enumi})}

\def\disp{\displaystyle}
\def\undertex#1{$\underline{\hbox{#1}}$}
\def\card{\mathop{\hbox{card}}}
\def\sgn{\mathop{\hbox{sgn}}}
\def\exp{\mathop{\hbox{exp}}}
\def\OFP{(\Omega,{\cal F},\PP)}
\newcommand\JM{Mierczy\'nski}
\newcommand\RR{\ensuremath{\mathbb{R}}}
\newcommand\CC{\ensuremath{\mathbb{C}}}
\newcommand\QQ{\ensuremath{\mathbb{Q}}}
\newcommand\ZZ{\ensuremath{\mathbb{Z}}}
\newcommand\NN{\ensuremath{\mathbb{N}}}
\newcommand\PP{\ensuremath{\mathbb{P}}}
\newcommand\abs[1]{\ensuremath{\lvert#1\rvert}}

\newcommand\normf[1]{\ensuremath{\lVert#1\rVert_{f}}}
\newcommand\normfRb[1]{\ensuremath{\lVert#1\rVert_{f,R_b}}}
\newcommand\normfRbone[1]{\ensuremath{\lVert#1\rVert_{f, R_{b_1}}}}
\newcommand\normfRbtwo[1]{\ensuremath{\lVert#1\rVert_{f,R_{b_2}}}}
\newcommand\normtwo[1]{\ensuremath{\lVert#1\rVert_{2}}}
\newcommand\norminfty[1]{\ensuremath{\lVert#1\rVert_{\infty}}}

\title{Diffusive KPP Equations with  Free Boundaries in Time Almost Periodic Environments: II. Spreading Speeds and Semi-Wave Solutions}

\author{Fang Li\\
School of Mathematical Sciences\\
University of Science and Technology of China\\
Hefei, Anhui, 230026, P.R.China\\
and\\
Department of Mathematics and Statistics\\
Auburn University\\
Auburn University, AL 36849 \\
 \\
 Xing Liang\\
 School of Mathematical Sciences\\
  University of Science and Technology of China\\
  Hefei, Anhui, 230026, P.R.China\\
   \\
  and
  \\
   Wenxian Shen\\
   Department of Mathematics and Statistics\\
Auburn University\\
Auburn University, AL 36849\\
 }

\date{}
\maketitle

\noindent \textbf{Abstract.} In this series of papers, we investigate  the spreading and
vanishing dynamics of time almost periodic diffusive KPP equations with  free boundaries.
Such equations are used to characterize the spreading of a new
 species in time almost periodic environments with free boundaries representing the spreading fronts.
In the first part of the series, we showed that a spreading-vanishing
dichotomy occurs for such free boundary problems (see \cite{LiLiSh1}). In this second part of the series,
we investigate the spreading speeds of such free boundary problems in the case that the spreading occurs. We first prove
the existence of a unique time almost periodic semi-wave solution associated to such a free boundary problem.
Using the semi-wave solution, we then prove that the free boundary problem
has a unique spreading speed.

\medskip


\medskip


\date{}
\maketitle


\section{Introduction}

This is the second part of a series of papers on  the spreading and
vanishing dynamics of  diffusive equations with free boundaries of the form,
\begin{equation}
\label{main-eq0}
\begin{cases}
u_t=u_{xx}+uf(t,x,u),\quad &t>0, 0<x<h(t)\cr
h^{'}(t)=-\mu u_x(t,h(t)),\quad &t>0\cr
u_x(t,0)=u(t,h(t))=0,\quad &t>0\cr
h(0)=h_0, u(0,x)=u_0(x),\quad &0\le x \le h_0,
\end{cases}
\end{equation}
where $\mu>0$.
We assume that
 $f(t,x,u)$ is a $C^1$ function in $t\in\RR$, $x\in\RR$, and
$u\in\RR$; $f(t,x,u)<0$ for $u\gg 1$; $f_u(t,x,u)<0$ for $u\ge 0$,
and $f(t,x,u)$ is almost periodic in $t$ uniformly with respect to
$x\in\RR$ and $u$ in bounded sets of $\RR$ (see (H1), (H2) in section 2 for detail). Here is a typical example
of such functions,
$f(t,x,u)=a(t,x)-b(t,x)u$, where $a(t,x)$ and $b(t,x)$ are almost periodic in $t$ and periodic in $x\in\RR$, and
$\inf_{t\in\RR,x\in\RR}b(t,x)>0$.

 Observe  that for given $h_0>0$ and $u_0$ satisfying
    \begin{equation}
 \label{initial-value}
 u_0\in C^2([0,h_0]),  \  u^{'}_0(0)=u_0(h_0)=0, \ {\rm and} \ u_0>0 \ {\rm in} \ [0,h_0),
 \end{equation}
   \eqref{main-eq0} has a (local) solution  $(u(t,\cdot;u_0,h_0)$, $h(t;u_0,h_0))$ with $u(0,\cdot;u_0,h_0)=u_0(\cdot)$ and $h_0(0;u_0,h_0)=h_0$ (see \cite{DuGuPe}). Moreover,
by comparison principle for parabolic equations, $(u(t,\cdot;u_0,h_0),h(t;u_0,h_0))$ exists for all
$t>0$ and $u_x(t,h(t;u_0,h_0);u_0,h_0)<0$. Hence $h(t;u_0,h_0)$ increases as $t$ increases.

Equation \eqref{main-eq0}  with $f(t,x,u)=u(a-bu)$ and $a$ and $b$ being
two positive constants was introduced by Du and Lin in \cite{DuLi}
to understand the spreading of species. A great deal of previous mathematical investigation on the spreading of species (in one space dimension case)
has been based on diffusive equations of the form
\begin{equation}
\label{kpp-general-eq}
u_t=u_{xx}+u f(t,x,u),\quad x\in\RR,
\end{equation}
where $f(t,x,u)<0$ for $u\gg 1$ and $f_u(t,x,u)<0$ for $u\ge 0$. Thanks to the pioneering works of Fisher (\cite{Fis}) and Kolmogorov, Petrowsky, Piscunov
(\cite{KPP}) on the following special case of \eqref{kpp-general-eq}
\begin{equation}
\label{kpp-special-eq}
u_t=u_{xx}+u(1-u),\quad x\in\RR,
\end{equation}
\eqref{main-eq0}, resp. \eqref{kpp-general-eq}, is referred to as diffusive Fisher or KPP equation.

One of the central problems for both \eqref{main-eq0} and \eqref{kpp-general-eq} is
to understand their spreading dynamics. For \eqref{kpp-general-eq}, this is closely related to
spreading speeds and transition fronts of \eqref{kpp-general-eq} and has been widely studied
(see \cite{BeHaNadir, LiZh, Nad, NoXi, Wei}, etc. for the study in the case that $f(t,x,u)$ is periodic in $t$ and/or $x$,
and see \cite{BeHa, BeHaNa, BeNa, HuSh, KoSh, NaRo, NoXi1, She1, She2, She3, TaZhZl, Zla}, etc. for the study in the case
that the dependence   of $f(t,x,u)$ on $t$ or $x$ is non-periodic). The spreading dynamics for \eqref{kpp-general-eq}
in many cases, including the cases that $f$ is periodic in $t$ and $x$,  is quite well understood.
For example, when $f(t,x,u)$ is periodic in $t$ and independent of $x$, or is independent of $t$ and periodic in $x$,
 it
has been proved
 that  \eqref{kpp-general-eq}
 has a unique positive periodic solution $u^*(t,x)$  which is asymptotically stable with
respect to periodic perturbations  and  has a spreading speed $c^*\in\RR$  in the sense that
 for any given $u_0\in C_{\rm unif}^b(\RR,\RR^+)$ with non-empty compact support,
 \begin{equation}
 \label{kpp-spreading-eq}
 \begin{cases}
 \lim_{|x| \le c^{'}t,t\to\infty}[u(t,x;u_0)-u^*(t,x)]=0\quad \forall\,\, c^{'}<c^*\cr
 \lim_{|x| \ge c^{''}t,t\to\infty} u(t,x;u_0)=0\quad \forall\,\, c^{''}> c^*,
 \end{cases}
 \end{equation}
 where $u(t,x;u_0)$ is the solution of \eqref{kpp-general-eq} with $u(0,x;u_0)=u_0(x)$
 (see \cite{LiZh, Wei}).

The spreading property \eqref{kpp-spreading-eq} for \eqref{kpp-general-eq} in the case that $f(t,x,u)$ is periodic in $t$ and independent of $x$
or independent of $t$ and periodic in $x$
 implies that spreading always happens for a solution of
\eqref{kpp-general-eq} with a positive initial data, no matter how small the positive initial data is.
The following strikingly different  spreading scenario has been proved for  \eqref{main-eq0} in the case that
$f(t,x,u)\equiv f(u)$ (see \cite{DuGu, DuLi}):  it exhibits a spreading-vanishing dichotomy in the sense that for any given positive constant $h_0$ and initial data
$u_0(\cdot)$ satisfying \eqref{initial-value}, either vanishing occurs (i.e. $\lim_{t\to\infty}h(t;u_0,h_0)<\infty$ and $\lim_{t\to\infty}u(t,x;u_0,h_0)=0$) or
spreading occurs (i.e. $\lim_{t\to\infty}h(t;u_0,h_0)=\infty$ and $\lim_{t\to\infty}u(t,x;u_0,h_0)=u^*$ locally uniformly in $x\in \RR^+$, where
$u^*$ is the unique positive solution of $f(u)=0$). The above spreading-vanishing dichotomy for \eqref{main-eq0} with
$f(t,x,u)\equiv f(u)$ has also been extended to the cases that $f(t,x,u)$ is periodic in $t$ or that $f(t,x,u)$ is independent of $t$ and periodic
in $x$ (see \cite{DuGuPe,DuLia}). The spreading-vanishing dichotomy proved for \eqref{main-eq0} in \cite{DuGu, DuGuPe, DuLia, DuLi}
 is well supported by some empirical evidences, for example, the introduction of several
bird species from Europe to North America in the 1900s was successful only after many initial
attempts (see \cite{LoHoMa, ShKa}).

While the spreading dynamics for \eqref{kpp-general-eq} with non-periodic time and/or space dependence has been
studied by many people recently
(see \cite{BeHa, BeHaNa, BeNa, HuSh, KoSh, NaRo, NoXi1, She1, She2, She3, TaZhZl, Zla}, etc.),  there is little study
on the spreading dynamics for \eqref{main-eq0} with non-periodic time and space dependence.

The objective of the current series
of papers is to investigate the spreading-vanishing dynamics of \eqref{main-eq0} in the case that $f(t,x,u)$ is almost periodic in $t$,
that is, to investigate  whether the population will
successfully establishes itself
in the entire space (i.e. spreading occurs), or  it fails to establish and vanishes
eventually (i.e. vanishing occurs).
Roughly speaking, for given $h_0>0$ and $u_0$ satisfying \eqref{initial-value},
if $h_\infty=\lim_{t\to\infty}h(t;u_0,h_0)=\infty$ and for any $M>0$,
$\liminf_{t\to\infty}\inf_{0\le x\le M}u(t,x;u_0,h_0)>0$,
we say {\it spreading} occurs. If $h_\infty<\infty$ and $\lim_{t\to\infty}u(t,x;u_0,h_0)=0$, we say {\it vanishing} occurs (see Definition
\ref{spreading-vanishing-def} for detail). We say a positive number $c^*$ is a {\it spreading speed} of \eqref{main-eq0} if for any $h_0>0$ and
 $u_0$ satisfying \eqref{initial-value} such that the spreading occurs,
$$\lim_{t\to\infty}\frac{h(t;u_0,h_0)}{t}=c^*$$
 and
 $$
 \liminf_{0\le x\le c^{'}t,t\to\infty}u(t,x;u_0,h_0)>0\quad \forall \,\, c^{'}<c^*
 $$
 (see Definition
\ref{spreading-vanishing-def} for detail).

The spreading speed of \eqref{main-eq0} is strongly related to the so called semi-wave solution of
 the following free boundary problem associated with \eqref{main-eq0},
\begin{equation}
\label{aux-main-eq3-0}
\begin{cases}
u_t=u_{xx}+uf(t,x,u),\quad -\infty<x<h(t)\cr
u(t,h(t))=0 \cr
h^{'}(t)=-\mu u_x(t,h(t)).
\end{cases}
\end{equation}
 If $(u(t,x),h(t))$ is an entire positive solution of \eqref{aux-main-eq3-0} with $\liminf_{x\to \infty}u(t,h(t)-x)>0$,
it is called a {\it semi-wave solution} of \eqref{aux-main-eq3-0}.

In the first part of the series of the papers, we studied the spreading and vanishing dichotomy for \eqref{main-eq0}.
Under proper assumptions (see (H1)-(H5) in Section 2 of part I, \cite{LiLiSh1}), we proved

\medskip
\noindent $\bullet$ {\it There are $l^*>0$ and a unique time almost periodic positive solution $u^*(t,x)$  of the following
 fixed boundary problem,
 \begin{equation}
 \label{aux-main-eq1}
 \begin{cases}
 u_t=u_{xx}+uf(t,x,u),\quad x>0\cr
 u_x(t,0)=0
 \end{cases}
 \end{equation}
 such that for any given $h_0>0$ and $u_0$ satisfying \eqref{initial-value}, either

\noindent (i) $h_\infty\le l^*$ and $u(t,x;u_0,h_0)\to 0$ as $t\to\infty$ or

 \noindent (ii) $h_\infty=\infty$ and $u(t,x;u_0,h_0)-u^*(t,x)\to 0$ as $t\to\infty$ locally uniformly in $x\ge 0$}
 (see \cite[Theorems 2.1 and 2.2]{LiLiSh1} or Proposition \ref{main-results-of-part1} in the case $f(t,x,u)\equiv f(t,u)$).

\medskip

In this second part of the series of papers, we study the existence of spreading speeds of \eqref{main-eq0} and
semi-wave solutions of \eqref{aux-main-eq3-0} in the case that $f(t,x,u)\equiv f(t,u)$, that is, we consider
\begin{equation}
\label{main-eq}
\begin{cases}
u_t=u_{xx}+uf(t,u),\quad &t>0, 0<x<h(t)\cr
h^{'}(t)=-\mu u_x(t,h(t)),\quad &t>0\cr
u_x(t,0)=u(t,h(t))=0,\quad &t>0\cr
h(0)=h_0, u(0,x)=u_0(x),\quad &0\le x \le h_0.
\end{cases}
\end{equation}
Note that \eqref{aux-main-eq3-0} then becomes
\begin{equation}
\label{aux-main-eq2}
\begin{cases}
u_t=u_{xx}+uf(t,u),\quad  -\infty<x<h(t)\cr
u(t,h(t))=0\cr
h^{'}(t)=-\mu u_x(t,h(t)).
\end{cases}
\end{equation}

To study the existence of spreading speeds of \eqref{main-eq} and
semi-wave solutions of \eqref{aux-main-eq2},
we  also consider the following  fixed boundary problem on half line,
\begin{equation}
\label{aux-main-eq3}
\begin{cases}
u_t=u_{xx}-\mu u_x(t,0)u_x(t,x)+uf(t,u),\quad 0<x<\infty\cr
u(t,0)=0.
\end{cases}
\end{equation}
Observe  that if $u^{*}(t,x)$ is an almost periodic positive solution of \eqref{aux-main-eq3} with $\liminf_{x\to\infty}u^{*}(t,x)>0$,
let $u^{**}(t,x)=u^{*}(t,h^{**}(t)-x)$ and $h^{**}(t)=\mu\int_0^ t u_x^*(s,0)ds$. Then
$(u^{**}(t,x),h^{**}(t))$ is an almost periodic semi-wave solution of \eqref{aux-main-eq2}. Hence  a  positive entire solution of \eqref{aux-main-eq3} gives
rise to a semi-wave solution of \eqref{aux-main-eq2}, and vice visa.
 Among others, we  prove

\medskip

\noindent $\bullet$ {\it There is a unique time almost periodic  stable positive solution $u^{*}(t,x)$ of \eqref{aux-main-eq3}
satisfying that $\liminf_{x\to\infty}u^{*}(t,x)>0$ and $u_x^{*}(t,0)>0$
(hence there is a time almost periodic semi-wave solution of \eqref{aux-main-eq2})} (see Theorem \ref{main-thm1}
for the detail).

\medskip

\noindent $\bullet$ {\it $c^{*}=\mu \lim_{t\to\infty}\frac{1}{t}\int_0^t u_x^{*}(s,0)ds$ is
the spreading speed of \eqref{main-eq}} (see Theorem \ref{main-thm2} for the detail).

\medskip

We remark that, when $f(t,u)$ is periodic in $t$ with period $T$, the authors of \cite{DuGuPe} used the following approach to prove
the existence of time periodic positive solution of  \eqref{aux-main-eq3}. First, for any given nonnegative time $T$-periodic function
$k(t)$ ($k(t+T)=k(t)$), they prove that there is a unique time $T$-periodic positive solution $U^{*}(t,x;k(\cdot))$ of the following
equation,
$$
\begin{cases}
u_t=u_{xx}-k(t)u_x+u f(t,u),\quad 0<x<\infty\cr
u(t,0)=0, u(t,x)=u(t+T,x).
\end{cases}
$$
Then by applying the Schauder fixed point theorem,  they prove that there is a nonnegative time $T$-periodic function $k^{*}(t)$ such that
$$
k^{*}(t)=\mu U_x^{*}(t,0;k^{*}(\cdot)).
$$
It then follows that $u^{*}(t,x)=U^{*}(t,x;k^{*}(\cdot))$ is a time $T$-periodic positive solution of \eqref{aux-main-eq3}.
The application of this approach to the time periodic case is  nontrivial. It
 is difficult to apply this approach to the case that $f(t,u)$ is almost periodic in $t$. We therefore prove
the existence of time almost periodic positive solution $u^{*}(t,x)$ directly. The proof is certainly also nontrivial and
can be applied to the time periodic case as well as more general time dependent cases.

We also remark that  similar results to the above hold for the
following double fronts free boundary problem:
\begin{equation}
\begin{cases}
\label{main-doub-eq}
u_t=u_{xx}+uf(t,u) \quad &t>0, g(t)<x<h(t) \cr
u(t,g(t))=0, g^{'}(t)=-\mu u_x(t,g(t)) \quad &t>0 \cr
u(t,h(t))=0, h^{'}(t)=-\mu u_x(t,h(t)) \quad &t>0,
\end{cases}
\end{equation}
where both $x=g(t)$ and $x=h(t)$ are to be determined.

 The rest of this paper is organized as follows.  In Section 2, we introduce the definitions and standing assumptions and
 state the main results of the paper. We present preliminary materials in Section 3 for the use in later sections.
 Section 4 is devoted to the investigation of time almost periodic KPP equations \eqref{aux-main-eq2} and
 \eqref{aux-main-eq3}  and Theorem \ref{main-thm1} is proved in this section.
  We show the existence and provide a characterization of the spreading speed of \eqref{main-eq}
 and prove Theorem \ref{main-thm2} in Section 5. The paper is ended up with some remarks on the spreading speeds and semi-wave
 solutions of \eqref{main-doub-eq} in Section 6.

\section{Definitions, Assumptions, and Main Results}

In this section, we introduce the  definitions and standing assumptions, and state the main results.

\subsection{Definitions and assumptions}

In this subsection, we introduce the definitions and standing assumptions. We first recall the definition of almost periodic
functions, next recall the definition of principal Lyapunov exponents for some linear parabolic equations, then state
the standing assumptions, and finally introduce the definition of spreading and vanishing for \eqref{main-eq}.

\begin{definition}[Almost periodic function]
\label{almost-periodic-def}
\begin{itemize}

\item[(1)] A continuous function $g:\RR\to \RR$ is called {\rm almost periodic} if  for any
$\epsilon>0$, the set
$$
T(\epsilon)=\{\tau\in\RR\,|\, |g(t+\tau)-g(t)|<\epsilon\,\, \text{for all}\,\, t\in\RR\}
$$
is relatively dense in $\RR$.

\item[(2)] Let $g(t,x,u)$ be a continuous function of $(t,x,u)\in\RR\times\RR^m\times\RR^n$. $g$ is said to be {\rm almost periodic in $t$ uniformly with respect to $x\in\RR^m$ and
$u$ in bounded sets} if
$g$ is uniformly continuous in $t\in\RR$,  $x\in\RR^m$, and $u$ in bounded sets and for each $x\in\RR^m$ and $u\in\RR^n$, $g(t,x,u)$ is almost periodic in $t$.

\item[(3)] For a given almost periodic function $g(t,x,u)$, the hull $H(g)$ is defined by
\begin{align*}
H(g)=\{\tilde g(\cdot,\cdot,\cdot)\,|\, & \exists t_n\to\infty \,\,\text{such that}\,\, g(t+t_n,x,u)\to \tilde g(t,x,u)\,\, \text{uniformly in}\,\, t\in\RR,\\
&
\,\, (x,u)\,\, \text{in bounded sets}\}.
\end{align*}
\end{itemize}
\end{definition}

\begin{remark}
\label{almost-periodic-rk}
(1) Let $g(t,x,u)$ be a continuous function of $(t,x,u)\in\RR\times\RR^m\times\RR^n$. $g$ is almost periodic in $t$ uniformly with respect to
$x\in\RR^m$ and $u$ in bounded sets  if and only if
 $g$ is uniformly continuous in $t\in\RR$,  $x\in\RR^m$, and $u$ in bounded sets and for any sequences $\{\alpha_n^{'}\}$,
$\{\beta_n^{'}\}\subset \RR$, there are subsequences $\{\alpha_n\}\subset\{\alpha_n^{'}\}$, $\{\beta_n\}\subset\{\beta_n^{'}\}$
such that
$$
\lim_{n\to\infty}\lim_{m\to\infty}g(t+\alpha_n+\beta_m,x,u)=\lim_{n\to\infty}g(t+\alpha_n+\beta_n,x,u)
$$
for each $(t,x,u)\in\RR\times\RR^m\times\RR^n$ (see \cite[Theorems 1.17 and 2.10]{Fink}).

(2) We may write $g(\cdot+t,\cdot,\cdot)$ as $g\cdot t(\cdot,\cdot,\cdot)$.
\end{remark}

For a given positive constant $l>0$ and a given $C^1$ function $a(t,x)$ with both $a(t,x)$ and $a_t(t,x)$ being
 almost periodic in $t$ uniformly in $x$ in bounded sets, consider
\begin{equation}
\label{linearized-eq1}
\begin{cases}
v_t=v_{xx}+a(t,x)v,\quad 0<x<l\cr v_x(t,0)=v(t,l)=0.
\end{cases}
\end{equation}

Let
$$
Y(l)=\{u\in C([0,l])\,|\, u(l)=0\}
$$
with the norm $\|u\|=\max_{x\in [0,l]}|u(x)|$ for $u\in Y(l)$.
Let $A=\Delta$ acting on $Y(l)$ with $\mathcal{D}(A)=\{u\in C^2([0,l])\cap Y(l)\,|\, u_x(0)=0\}$.
Note that  $A$ is a sectorial operator. Let $0<\alpha<1$ be such that $\mathcal{D}(A^\alpha)\subset C^1([0,l])$. Fix such an $\alpha$.
Let
\begin{equation}
\label{bounded-domain-space-eq0}
X(l)=\mathcal{D}(A^\alpha).
\end{equation}
Then $X(l)$ is strongly ordered Banach spaces with positive cone
$$
X^+(l)=\{u\in X(l)\,|\, u(x)\ge 0\}.
$$
Let
$$
X^{++}(l)={\rm Int}(X^+(l)).
$$
If no confusion occurs, we may write $X(l)$ as $X$.

By semigroup theory (see \cite{PA}), for any  $v_0\in X(l)$, \eqref{linearized-eq1}
has a unique solution $v(t,\cdot;v_0,a)$ with $v(0,\cdot;v_0,a)=v_0(\cdot)$.

For  given  constants $l>0$, $\gamma\ge 0$,  and a given $C^1$ function
 $a(t,x)$ with both $a(t,x)$ and $a_t(t,x)$ being almost periodic function in $t$ uniformly in $x$ in bounded sets, consider
also
\begin{equation}
\label{aaux-linearized-eq1}
\begin{cases}
v_t=v_{xx}-\gamma v_x+a(t,x)v,\quad 0<x<l\cr
v(t,0)=v(t,l)=0.
\end{cases}
\end{equation}
Let
$$
\tilde Y(l)=\{u\in C([0,l])\,|\, u(0)=u(l)=0\}.
$$
Let $A=\Delta$ acting on $\tilde Y(l)$ with $\mathcal{D}(A)=\{u\in C^2([0,l])\cap \tilde  Y(l)\}$.
Note that  $A$ is a sectorial operator. Let $0<\alpha<1$ be such that $\mathcal{D}(A^\alpha)\subset C^1([0,l])$. Fix such an $\alpha$.
Let
\begin{equation}
\label{bounded-domain-space-eq1}
\tilde X(l)=\mathcal{D}(A^\alpha).
\end{equation}
Then, for any $v_0\in \tilde X(l)$, \eqref{aaux-linearized-eq1} has a unique solution
$\tilde v(t,\cdot;v_0,a)$ with $\tilde v(0,\cdot;v_0,a)=v_0(\cdot)$.

\begin{definition}[Principal Lyapunov exponent]
\label{principal-lyapunov-exp}
\begin{itemize}
\item[(1)] Let $V(t,a)v_0=v(t,\cdot;v_0,a)$ for $v_0(\cdot)\in X(l)$ and
$$
\lambda(a,l)=\limsup_{t\to\infty}\frac{\ln \|V(t,a)\|_{X(l)}}{t}.
$$
$\lambda(a,l)$ is called the {\rm principal Lyapunov exponent} of \eqref{linearized-eq1}.

\item[(2)] Let
$$
\tilde\lambda(a,\gamma,l)=\limsup_{t\to\infty}\frac{\ln\|\tilde V(t,a)\|_{\tilde X(l)}}{t}
$$
where $\tilde V(t,a)v_0=\tilde v(t,\cdot;v_0,a)$ for $v_0\in\tilde X(l)$.
$\tilde\lambda(a,\gamma,l)$ is called the {\rm principal Lyapunov exponent} of \eqref{aaux-linearized-eq1}.
\end{itemize}
\end{definition}

Let (H1)-(H3) be the following standing assumptions.

\medskip

\noindent {\bf (H1)} {\it $f(t,u)$ is $C^1$ in $(t,u)\in\RR^2$, $Df=(f_t,f_u)$ is bounded in $t\in\RR$ and
in $u$ in bounded sets, and $f$ is monostable in $u$ in the sense that
there are $M>0$ such that $$\sup_{t\in\RR,u\ge M}f(t,u)<0$$ and
$$\sup_{t\in\RR,u\ge 0}f_u(t,u)<0.$$}

\medskip

\noindent{\bf (H2)} {\it $f(t,u)$ and $Df(t,u)=(f_t(t,u),f_u(t,u))$ are  almost periodic in $t$ uniformly with respect to
$u$ in bounded sets.
}

\medskip
\noindent {\bf (H3)}  {\it $\lim_{t\to\infty}\frac{1}{t}\int_0^t f(s,0)ds>0$.}

\medskip

Assume (H1) and (H2).
We remark that (H3) implies that there are $L^*\ge l^*>0$ such that $\lambda(a(\cdot),l)>0$ for $l>l^*$
and $\tilde\lambda(a(\cdot),0, l)>0$ for $l>L^*$, where $a(t)=f(t,0)$ (see Lemma \ref{lyapunov-exponent-lm3}
and Remark \ref{Lyapunov-exp-rk}
for the reasonings).

Consider \eqref{main-eq}. Throughout this paper, we assume (H1)-(H3).
For any given $h_0>0$ and $u_0$ satisfying \eqref{initial-value}, \eqref{main-eq} has a unique solution
$(u(t,x;u_0,h_0),h(t;u_0,h_0))$ with $u(0,x;u_0,h_0)=u_0(x)$
and $h(0;u_0,h_0)=h_0$ (see \cite{DuGuPe}). By comparison principle for parabolic equations, $u(t,x;u_0,h_0)$ exists for all $t>0$ and
$u_x(t,h(t;u_0,h_0);u_0,h_0)< 0$ for $t>0$.
Hence $h(t;u_0,h_0)$ is monotonically increasing, and therefore
there exists $h_{\infty}\in(0,+\infty]$ such that
$\lim_{t\to+\infty}h(t;u_0,h_0)=h_{\infty}$.

\begin{definition}[Spreading-vanishing and spreading speed]
\label{spreading-vanishing-def}
Consider \eqref{main-eq}.
\begin{itemize}
\item[(1)] For given $h_0>0$ and $u_0$ satisfying \eqref{initial-value}, let $h_\infty=\lim_{t\to\infty}h(t;u_0,h_0)$.
 It is said that the {\rm vanishing occurs} if $ h_\infty<\infty$ and
$\lim_{t\to\infty}\|u(t,\cdot;u_0,h_0)\|_{C([0,h(t)])}=0$. It is said that the {\rm spreading occurs}
 if $h_\infty=\infty$ and
 $\liminf_{t\to\infty} u(t,x;u_0,h_0)>0$ locally uniformly in
$x\in [0,\infty)$.

\item[(2)] A real number $c^*>0$ is called the {\rm spreading speed} of \eqref{main-eq} if for any $h_0>0$
and $u_0$ satisfying \eqref{initial-value} such that the spreading occurs,
there holds
$$
\lim_{t\to\infty}\frac{h(t;u_0,h_0)}{t}=c^*
$$
and
$$
\liminf_{0\le x\le c^{'}t,t\to\infty}u(t,x;u_0,h_0)>0,\quad \forall\,\, c^{'}<c^*.
$$
\end{itemize}
\end{definition}

Assume (H1)-(H3). It is known that
there is a unique time almost periodic positive solution  $V^*(t)$ of the following ODE (see Lemma \ref{ode-lm}),
\begin{equation*}
u_t=uf(t,u).
\end{equation*}

\begin{definition}
\label{semi-wave-def}
An entire positive solution $(u(t,x),h(t))$ of \eqref{aux-main-eq2} is called an almost periodic semi-wave solution
if $u(t,h(t)-x)$ is almost periodic in $t$ uniformly with respect to $x\ge 0$ and $h^{'}(t)$ is almost periodic in $t$, and $\lim_{x\to \infty}u(t,h(t)-x)=V^*(t)$
uniformly in $t\in\RR$.
\end{definition}

\begin{remark}
\label{semi-wave-rk}
If $ \tilde u^{**}(t,x)$ is an almost periodic positive solution of \eqref{aux-main-eq3} uniformly with respect to $x\ge 0$
and  $\lim_{x\to\infty}\tilde u^{**}(t,x)=V^*(t)$
uniformly in $t$, then $( u^{**}(t,x), h^{**}(t))$ is an almost periodic semi-wave solution of \eqref{aux-main-eq2},
where $ u^{**}(t,x)=\tilde u^{**}(t, h^{**}(t)-x)$ and $ h^{**}(t)=\mu\int_0^t \tilde u^{**}_x(s,0)ds$.
\end{remark}

\subsection{Main results}

In this subsection, we state the main results of this paper.
To do so, we first recall the main results obtained in the part I of the series.

\begin{proposition}
 \label{main-results-of-part1}
Assume (H1)-(H3).
 For any given $h_0>0$ and $u_0$ satisfying \eqref{initial-value}, let $(u(t,x;u_0,h_0),h(t;u_0,h_0))$ be
 the solution of \eqref{main-eq} with $(u(0,x;u_0,h_0),h(0;u_0,h_0))=(u_0(x),h_0)$. Then either
\begin{itemize}
\item[(i)]  $h_\infty\le l^*$ and $u(t,x;u_0,h_0)\to 0$ as $t\to\infty$ or

 \item[(ii)]  $h_\infty=\infty$ and $u(t,x;u_0,h_0)-V^*(t)\to 0$ as $t\to\infty$ locally uniformly in $x\ge 0$.
\end{itemize}
\end{proposition}

\begin{proof}
See \cite[Theorem 2.2]{LiLiSh1}.
\end{proof}

The main results  of this paper are stated in the following two theorems.

\begin{theorem}[Almost periodic semi-waves]
\label{main-thm1}
 Assume (H1)-(H3).
 \begin{itemize}
 \item[(1)] There is a  time almost periodic
 solution $\tilde u^{**}(t,x)$ of \eqref{aux-main-eq3} with $\lim_{x\to\infty}\tilde u^{**}(t,x)=V^*(t)$ uniformly in $t\in\RR$ and hence
  there is a   time almost periodic positive semi-wave solution $(u^{**}(t,x), h^{**}(t))$ of \eqref{aux-main-eq2}
 with $h^{**}(0)=0$.

 \item[(2)] If $\tilde u_1^{**}(t,x)$ and $\tilde u_2^{**}(t,x)$ are two almost periodic positive solutions of \eqref{aux-main-eq3}
 satisfying that $\lim_{x\to\infty}\tilde u_i^{**}(t,x)=V^*(t)$ uniformly in $t\in\RR$ ($i=1,2$), then $\tilde u_1^{**}(t,x)\equiv \tilde  u_2^{**}(t,x)$.

 \item[(3)] For any bounded positive solution $\tilde u(t,x)$ of \eqref{aux-main-eq3} with $\liminf_{x\to\infty}\inf_{t\ge 0}\tilde u(t,x)>0$,
 $$
 \lim_{t\to\infty} [\tilde u^{**}(t,x)-\tilde u(t,x)]=0
 $$
 uniformly in $x\ge 0$.
 \end{itemize}
\end{theorem}

\begin{theorem}[Spreading speed and semi-wave]
\label{main-thm2}
Assume (H1)-(H3) and $f(t,x,u)\equiv f(t,u)$.
Let $(u^{**}(t,x),h^{**}(t))$ be as in Theorem \ref{main-thm1} (1),
   and
  $$c^*=\lim_{t\to\infty}\frac{h^{**}(t)}{t}.
  $$
Then $c^*$ is the spreading speed of \eqref{main-eq}, that is,
 for any given $h_0>0$ and $u_0$ satisfying \eqref{initial-value}, if $h_\infty=\lim_{t\to\infty}h(t;u_0,h_0)=\infty$, then
$\lim_{t\to\infty}\frac{h(t;u_0,h_0)}{t}=c^*$ and
$$\lim_{t\to\infty}\max_{x\leq(c^*-\epsilon)t}|u(t,x;u_0,h_0)-V^*(t)|=0$$
for every small $\epsilon>0$.
\end{theorem}

\section{Preliminary}

In this section,
we present some preliminary results to be applied in later sections, including basic properties for principal Lyapunov exponents (see subsection 3.1),
the asymptotic
dynamics of some  diffusive KPP equations  with time almost periodic dependence in fixed environments (see subsection 3.2),
 and comparison principles
for free boundary problems (see subsection 3.3).

\subsection{Principal Lyapunov exponents}

Consider \eqref{linearized-eq1}. Let $X=X(l)$, where $X(l)$ is as in \eqref{bounded-domain-space-eq0}. We denote
by $\|\cdot\|$ the norm in $X$ or in $\mathcal{L}(X,X)$.
Recall that for any $v_0\in X$, \eqref{linearized-eq1} has a unique solution $v(t,\cdot;v_0,a)$ and
$$
\lambda(a,l)=\limsup_{t\to\infty}\frac{\ln\|V(t,a)\|}{t}
$$
where $V(t,a)v_0=v(t,\cdot;v_0,a)$.
For any $b\in H(a)$, consider also
\begin{equation}
\label{linearized-eq1-1}
\begin{cases}
v_t=v_{xx}+b(t,x)v,\quad 0<x<l\cr
v_x(t,0)=v(t,l)=0,
\end{cases}
\end{equation}
For any $v_0\in X$,
\eqref{linearized-eq1-1} has also a unique solution $v(t,\cdot;v_0,b)$ with $v(0,\cdot;v_0,b)=v_0$.

\begin{lemma}
\label{lyapunov-exponent-lm0}
There is $\phi_l:H(a)\to X^{++}$   satisfying the following properties.
\begin{itemize}
\item[(i)] $\|\phi_l(b)\|=1$ for any $b\in H(a)$ and $\phi_l:H(a)\to X^{++}$ is continuous.

\item[(ii)] $v(t,\cdot;\phi_l(b),b)=\|v(t,\cdot;\phi_l(b),b)\|\phi_l(b(\cdot+t,\cdot))$.

\item[(iii)] $\lim_{t\to\infty}\frac{\ln \|v(t,\cdot;\phi_l(b),b)\|}{t}=\lambda(a,l)$ uniformly in $b\in H(a)$.
\end{itemize}
\end{lemma}

\begin{proof}
It follows from \cite{MiSh1} (see also \cite{MiSh3, ShYi}).
\end{proof}

\begin{lemma}
\label{lyapunov-exponent-lm3} Suppose that $a(t,x)\equiv a(t)$. Then
$$
\lambda(a,l)=\hat a+\lambda_0(l),
$$
where $\hat a=\lim_{t\to\infty}\frac{1}{t}\int_0^t a(s)ds$ and $\lambda_0(l)$ is the principal
eigenvalue of
\begin{equation}
\label{egv-eq}
\begin{cases}
u_{xx}=\lambda u,\quad 0<x<l\cr
u_x(0)=u(l)=0.
\end{cases}
\end{equation}
\end{lemma}

\begin{proof}
Let $\tilde v(t,x)=v(t,x)e^{-\int_0^t a(s)ds}$.  Then \eqref{linearized-eq1} becomes
$$
\begin{cases}
\tilde v_t=\tilde v_{xx},\quad 0<x<l\cr
\tilde v_x(t,0)=\tilde v(t,l)=0.
\end{cases}
$$
It then follows that $\lambda(a,l)=\hat a+\lambda(0,l)$. It is clear that
$\lambda(0,l)=\lambda_0(l)$. The lemma then follows.
\end{proof}

\begin{remark}
\label{Lyapunov-exp-rk}
(1) Principal Lyapunov exponent theory for \eqref{linearized-eq1} also holds for \eqref{aaux-linearized-eq1}.

\medskip

(2)
When $a(t,x)\equiv a(t)$,  $\tilde\lambda(a,\gamma,l)=\hat a+\tilde \lambda(0,\gamma,l)$. Note that
$\lambda(0,l)=-\frac{\pi^2}{4 l^2}$ and $\tilde \lambda(0,\gamma,l)=-\Big(\frac{\gamma^2}{4}+\frac{\pi^2}{l^2}\Big)$.
Hence $\lambda(a,l)>0$ and $\tilde\lambda(a,0,l)>0$ for $l\gg 1$.
\end{remark}

\subsection{Asymptotic dynamics of diffusive KPP equations  with time almost periodic dependence
in fixed domains}

In this subsection, we consider the asymptotic dynamics of the following
 KPP equations,
\begin{equation}
\label{ode-eq1}
u_t=uf(t,u),
\end{equation}
\begin{equation}
\label{kpp-eq1}
\begin{cases} u_t=u_{xx}+uf(t,u),\quad x>0\cr
 u_x(t,0)=0,
\end{cases}
\end{equation}
and
\begin{equation}
\label{kpp-eq2}
\begin{cases}
u_t=u_{xx}-\epsilon \mu u_x+uf(t,u),\quad x>0\cr
u(t,0)=0.
\end{cases}
\end{equation}

Throughout this subsection, we assume (H1) and (H2).
Let
$$
H(f)={\rm cl}\{f(\cdot+\tau,\cdot)\,|\, \tau\in\RR\}
$$
where the closure is taken in the open compact topology. Observe that for any $g\in H(f)$, $g$ also satisfies (H1) and (H2).

First of all, consider \eqref{ode-eq1} and
\begin{equation}
\label{ode-eq1-hull}
u_t=ug(t,u)
\end{equation}
for any $g\in H(f)$. By fundamental theory for ordinary differential equations,
 for any $u_0\in\RR$ and $g\in H(f)$, \eqref{ode-eq1-hull} has a unique (local) solution
$u(t;u_0,g)$ with $u(0;u_0,g)=u_0$. By (H1), for any $u_0\ge 0$, $u(t;u_0,g)\ge 0$ and $u(t;u_0,g)$ exists for all $t\ge 0$.

\begin{lemma}
\label{ode-lm}
For any $g\in H(f)$, there is a unique stable almost periodic  positive solution $u_g(t)$ of \eqref{ode-eq1-hull}.
\end{lemma}

\begin{proof}
It follows from \cite[Theorem 4.1]{ShYi0}.
\end{proof}

Next, consider \eqref{kpp-eq1} and
\begin{equation}
\label{kpp-eq1-hull}
\begin{cases} u_t=u_{xx}+ug(t,u),\quad x>0\cr
 u_x(t,0)=0.
\end{cases}
\end{equation}
for any $g\in H(f)$.

Let
$$
\tilde X_0= C_{\rm unif}^b([0,\infty))
$$
with the norm $\|u\|=\sup_{x\in[0,\infty)}|u(x)|$ for $u\in \tilde X_0$. The operator
$A=\Delta$ with $\mathcal{D}(A)=\{u\in  \tilde X_0\,|\, u^{'}(\cdot),u^{''}(\cdot)\in \tilde X_0,\, u^{'}(0)=0\}$ is a sectorial operator. Let
\begin{equation}
\label{tilde-x-space}
\tilde X= \text{a fractional power space of}\,\, A\,\, \text{such that for any}\,\, u\in \tilde X,\,\, u^{'}(\cdot)\in C_{\rm unif}^b([0,\infty)).
\end{equation}
Let
$$
\tilde X^+=\{u\in \tilde X\,|\, u(x)\ge 0\quad {\rm for}\quad x\in\RR^+\}
$$
and
$$
\tilde X^{++}=\{u\in \tilde X^+\,|\, \inf_{x\ge 0}u(x)>0\}.
$$

By general semigroup theory, for any $u_0\in \tilde X$ and any $g\in H(f)$, there is a
unique (local) solution $u(t,\cdot;u_0,g)$ with
$u(0,x;u_0,g)=u_0(x)$. By comparison principle for parabolic equations, for any $u_0\in\tilde X^+$,
$u(t,\cdot;u_0,g)\in\tilde X^+$ and $u(t,\cdot;u_0,g)$ exists for all $t\ge 0$. If $u_0\in\tilde X^{++}$, then
$u(t,\cdot;u_0,g)\in\tilde X^{++}$ for all $t\ge 0$.

\begin{remark}
\label{kpp-rk1}
For any $g\in H(f)$, $u(t,x)=u_g(t)$ is an almost periodic solution of \eqref{kpp-eq1-hull}. Moreover, for any
$u_0\in \tilde X^{++}$,
$$
u(t,x;u_0,g)-u_g(t)\to 0
$$
as $t\to\infty$ uniformly in $x\ge 0$.
\end{remark}

Consider now \eqref{kpp-eq2} and
\begin{equation}
\label{kpp-eq2-hull}
\begin{cases}
u_t=u_{xx}-\epsilon \mu u_x+ug(t,u),\quad x>0\cr
u(t,0)=0
\end{cases}
\end{equation}
for any $g\in H(f)$.

Let
$$
\hat X_0=\{u\in C_{\rm unif}^b([0,\infty))\,|\, u(0)=0\}
$$
with the norm $\|u\|=\sup_{x\in[0,\infty)}|u(x)|$ for $u\in \hat X_0$. The operator
$A=\Delta$ with $\mathcal{D}(A)=\{u\in \hat X_0\,|\, u^{'}(\cdot),u^{''}(\cdot)\in C_{\rm unif}^b([0,\infty))\}$ is a sectorial operator. Let
\begin{equation}
\label{hat-x-space}
\hat X= \text{a fractional power space of}\,\, A\,\, \text{such that for any}\,\, u\in \hat X,\,\, u^{'}(\cdot)\in C_{\rm unif}^b([0,\infty)).
\end{equation}
Let
$$
\hat X^+=\{u\in \hat X\,|\, u(x)\ge 0\quad {\rm for}\quad x\in\RR^+\}
$$
and
$$
\hat X^{++}=\{u\in \hat X^+\,|\, \inf_{x\ge \epsilon}u(x)>0\quad\text{for any}\,\,\epsilon>0\,\,{\rm and}\,\, u^{'}(0)>0\}.
$$

 By general semigroup theory, for any $u_0\in \hat X$ and any $g\in H(f)$, there is a
unique (local) solution $u(t,\cdot;u_0,g)$ with
$u(0,x;u_0,g)=u_0(x)$. By comparison principle for parabolic equations, for any $u_0\in\hat X^+$,
$u(t,\cdot;u_0,g)\in\hat X^+$ and $u(t,\cdot;u_0,g)$ exists for all $t\ge 0$. If $u_0\in\hat X^{++}$, then
$u(t,\cdot;u_0,g)\in\hat X^{++}$ for all $t\ge 0$.

By Remark \ref{Lyapunov-exp-rk},
there are  $\tilde l^*>0$ and $\epsilon^*>0$ such that
$\tilde\lambda(a,\epsilon\mu,l)>0$ for $l>\tilde l^*$ and $0<\epsilon<\epsilon^*$.


\begin{lemma}
\label{semi-wave-lm1-1}
Let $\epsilon>0$ be given such that $\tilde\lambda(a,\epsilon\mu,l)>0$ for $l\gg 1$.
For given $g\in H(f)$,
Consider
\begin{equation}
\label{aux-semi-wave-eq1}
\begin{cases}
u_t=u_{xx}-\epsilon \mu u_x+ ug(t,u),\quad x>0\cr
u(t,0)=0.
\end{cases}
\end{equation}
For any $u_0\in \hat X^{++}$, $\inf_{t\ge 0, g\in H(f)}\p_x u_{\epsilon}(t,0;u_0,g)>0$.
\end{lemma}

\begin{proof}
First of all, we consider the following problem
\begin{equation}
\label{aux-linear-semi-wave-eq1}
\begin{cases}
u_t=u_{xx}-\epsilon \mu u_x+ ug(t,u),\quad 0<x<l\cr
u(t,0)=u(t,l)=0
\end{cases}
\end{equation}
Since $\tilde\lambda(a,\epsilon\mu,l)>0$, there is a unique time almost periodic positive stable
solution $u^l_{\epsilon,g}(t,x)$ of \eqref{aux-linear-semi-wave-eq1}. Moreover, for any
$\tilde u_0\in C([0,l])$ with $\tilde u_0(0)=\tilde u_0(l)=0$ and $\tilde u_0(x)>0$ for $x\in (0,l)$,
$$
\lim_{t\to\infty}|u^l_\epsilon(t,x;\tilde u_0,g)-u^l_{\epsilon,g}(t,x)|=0
$$
uniformly in $x\in [0,l]$ and $g\in H(f)$,
and
$$
\lim_{t\to\infty}|\p_x u^l_\epsilon(t,0;\tilde u_0,g)-\p_x u^l_{\epsilon,g}(t,0)|=0
$$
uniformly in $g\in H(f)$, where $u^l_\epsilon(t,x;\tilde u_0,g)$ is the solution of 
\eqref{aux-linear-semi-wave-eq1} with $u^l_\epsilon(0,x;\tilde u_0,g)=\tilde u_0(x)$.

Now for any $u_0\in \hat X^{++}$, choose $\tilde u_0\in C([0,l])$ such that $\tilde u_0(0)=\tilde u_0(l)=0$,
$\tilde u_0(x)>0$ for $x\in (0,l)$, $\p_x u_0(0)>0$, and 
$$
\tilde u_0(x)\le u_0(x)\quad {\rm for}\quad 0\le x\le l.
$$
Then by comparison principle for parabolic equations, we have
$$
u_\epsilon(t,x;u_0,g)\ge u^l_\epsilon(t,x;\tilde u_0,g)\quad {\rm for}\quad 0\le x\le l.
$$
This implies that
$$
\p_x u_\epsilon(t,0;u_0,g)\ge \p_x u^l_\epsilon(t,0;\tilde u_0,g)
$$
and then
$$
\inf_{t\ge 0,g\in H(f)}\p_x u_\epsilon(t,0;u_0,g)>0.
$$
This proves the lemma.
\end{proof}

\subsection{Comparison principal for free boundary problems}

In order for later application, we present some comparison principles for free boundary problems
in this subsection.

\begin{proposition}
\label{comparison-principle} Let $f(t,u)$ be a function satisfying
(H1) and (H2). Suppose that $T\in (0,\infty)$, $\bar{h}
\in C^{1}([0,T])$, $\bar{u}\in C(\bar{D}^{*}_{T})\cap C ^{1,2}(D^{*}_{T})$
with $D^{*}_{T}=\{(t,x)\in\R^{2}:0<t\leq T, 0<x<\bar{h}(t)\}$, and
\begin{equation*}
\begin{cases}
\bar{u}_t\geq\bar{u}_{xx}+\bar{u}f(t,\bar{u}),\quad &t>0, 0<x<\bar{h}(t)
\cr \bar{h}^{'}(t)\geq-\mu\bar{u}_x(t,\bar{h}(t)),\quad &t>0, \cr
\bar{u}_x(t,0)\leq 0,u(t,\bar{h}(t))=0,\quad  &t>0
\end{cases}
\end{equation*}
If $h_{0}\leq \bar{h}(0)$ and $u_{0}(x)\leq\bar{u}(0,x)$ in $[0,h_{0}]$,
then the solution $(u,h)$ of the free boundary problem \eqref{main-eq} satisfies
$$h(t)\leq \bar{h}(t)  \ for\ all\ t\in(0,T],\ u(t,x)\leq\bar{u}(t,x) \ for \
t\in(0,T]\ and\ x\in(0,h(t)).$$
\end{proposition}

\begin{proof}
The proof of this Proposition is similar to that of Lemma 3.5 in
\cite{DuLi} and Lemma 2.6 in \cite{DuGu}.
\end{proof}

\begin{remark}
\label{re1}
The pair $(\bar{u},\bar{h})$ in Proposition \ref{comparison-principle}
is called an upper solution of the free boundary problem. We can
define a lower solution by reversing all the inequalities in the
obvious places.

\end{remark}

\begin{proposition}
\label{comparison-principle1} Let $f(t,u)$ be a function satisfying
(H1) and (H2). Suppose that $T\in (0,\infty)$, $\bar{h}
\in C^{1}([0,T])$, $\bar{u}\in C ^{1,2}(D^{*}_{T})$
with $D^{*}_{T}=\{(t,x)\in\R^{2}:0<t\leq T, -\infty<x<\bar{h}(t)\}$, and
\begin{equation*}
\begin{cases}
\bar{u}_t\geq\bar{u}_{xx}+\bar{u}f(t,\bar{u}),\quad &t>0, -\infty<x<\bar{h}(t)
\cr \bar{h}^{'}(t)\geq-\mu\bar{u}_x(t,\bar{h}(t)),\quad &t>0, \cr
u(t,\bar{h}(t))=0,\quad  &t>0
\end{cases}
\end{equation*}
If $h_{0}\leq \bar{h}(0)$ and $u_{0}(x)\leq\bar{u}(0,x)$ in $(-\infty,h_{0}]$,
then the solution $(u,h)$ of the free boundary problem
\begin{equation*}
\begin{cases}
u_t=u_{xx}+uf(t,u),\quad &t>0, -\infty<x<h(t)\cr h^{'}(t)=-\mu u_x(t,h(t)),
\quad &t>0 \cr u(t,h(t))=0,\quad &t>0 \cr h(0)=h_0, u(0,x)=u_{0}(x),\quad &-\infty<x\leq h_{0}
\end{cases}
\end{equation*}
satisfies
$$h(t)\leq \bar{h}(t)  \ for\ all\ t\in(0,T],\ u(t,x)\leq\bar{u}(x,t) \ for \
t\in(0,T]\ and\ x\in(-\infty,h(t)).$$

\end{proposition}

\begin{proof}
The proof of this Proposition is similar to Proposition \ref{comparison-principle}.
\end{proof}

\begin{proposition}
\label{global-existence} For any given $h_0>0$ and $u_0$ satisfying \eqref{initial-value},
$(u(t,x;u_0,h_0), h(t;u_0,h_0))$ exists for all $t\ge 0$.
\end{proposition}

\begin{proof}
The proof is similar to that of Theorem 4.3 in \cite{DuGu}.
\end{proof}

\begin{remark}
From the uniqueness of the solution to \eqref{main-eq} and
some standard compactness argument, we can obtain that the
unique solution $(u,h)$ depends continuously on $u_{0}$ and
the parameters appearing in \eqref{main-eq}.
\end{remark}

We will need some simple variants of Proposition \ref{comparison-principle}
and Remark \ref{re1}, whose proofs are similar to the original ones
and therefore omitted.

\begin{lemma}
\label{com-principle free boun 2}
Let $f(t,u)$ be a function satisfying (H1) and (H2).
Suppose that $T\in(0,\infty)$, $\bar{h}\in C^1([0,T])$,
$\bar{u}\in C^{1,2}(D_T^*)$ with
$D_T^*=\{(t,x)\in\RR^2: 0\leq t \leq T, \ 0\leq x \leq \bar{h}(t)\}$,
and
\begin{equation*}
\begin{cases}
\bar{u}_t\geq \bar{u}_{xx}+\bar uf(t,\bar{u}),
\quad &t\in(0,T], 0<x<\bar{h}(t), \cr
\bar{u}(t,\bar{h}(t))=0,  {\bar{h}}^{'}(t)\geq -\mu{\bar{u}}_x(t,\bar{h}(t)),
\quad &t\in(0,T], \cr
\bar{u}(t,0)\geq l(t), \quad &t\in(0,T].
\end{cases}
\end{equation*}
If $h\in C^1([0,T])$ and $u\in C^{1,2}(D_T)$ with
$D_T=\{(t,x)\in\RR^2: 0\leq t \leq T, \ 0\leq x \leq h(t)\}$
satisfy
$$0<h(0)\leq\bar{h}(0), 0<u(0,x)\leq\bar{u}(0,x) \ for \ 0\leq x \leq h(0), $$
and
\begin{equation}
\begin{cases}
\label{free-boun-nonhom}
u_t=u_{xx}+uf(t,u),
\quad &t\in(0,T], 0<x<h(t), \cr
u(t,h(t))=0, \ h^{'}(t)=-\mu u_x(t,h(t)),
\quad &t\in(0,T], \cr
u(t,0)=l(t), \quad &t\in(0,T].
\end{cases}
\end{equation}
then
$$h(t)\leq \bar{h}(t) \ for \ t\in(0,T], \ u(t,x)\leq\bar{u}(t,x)
\ for \ (t,x)\in(0,T]\times(0,h(t)).$$
\end{lemma}

Similarly, we have the following analogue of Lemma
\ref{com-principle free boun 2}.

\begin{lemma}
\label{com-principle free boun 3}
Let $f(t,u)$ be as in Lemma
\ref{com-principle free boun 2}.
Suppose that $T\in(0,\infty)$, $\underline{h}\in C^1([0,T])$,
$\underline{u}\in C^{1,2}(D^+_T)$ with
$D_T^+=\{(t,x)\in\RR^2: 0\leq t \leq T, \ 0\leq x \leq \underline{h}(t)\}$,
and
\begin{equation*}
\begin{cases}
\underline{u}_t\leq \underline{u}_{xx}+\underline u f(t,\underline{u}),
\quad &t\in(0,T], 0<x<\underline{h}(t), \cr
\underline{u}(t,\underline{h}(t))=0, \ \underline{h}^{'}(t)\leq -\mu\underline{u}_x(t,\underline{h}(t)),
\quad &t\in(0,T], \cr
\underline{u}(t,0)\leq l(t), \quad &t\in(0,T].
\end{cases}
\end{equation*}
If $h\in C^1([0,T])$ and $u\in C^{1,2}(D^+_T)$ satisfy
\eqref{free-boun-nonhom} and
$$h(0)\geq\underline{h}(0), u(0,x)\geq\underline{u}(0,x)\geq 0,  \ for \ 0\leq x \leq\underline{h}(0),$$
then
$$h(t)\geq \underline{h}(t) \ for \ t\in(0,T], \ u(t,x)\geq\underline{u}(t,x)
\ for \ (t,x)\in(0,T]\times(0,\underline{h}(t)).$$
\end{lemma}

\section{Basic Properties of Diffusive KPP Equations in Unbounded Domains}

In this section, we present some basic properties of  \eqref{aux-main-eq2} and \eqref{aux-main-eq3}.
Throughout this subsection, we assume (H1) and (H2).
Let
$$
H(f)={\rm cl}\{f(\cdot+\tau,\cdot)\,|\, \tau\in\RR\},
$$
where the closure is taken in the open compact topology. Observe that for any $g\in H(f)$, $g$ also satisfies (H1) and (H2).

Consider \eqref{aux-main-eq3} and
\begin{equation}
\label{aux-main-eq3-hull}
\begin{cases}
u_t=u_{xx}-\mu u_x(t,0) u_x(t,x)+ u g(t,u),\quad 0<x<\infty\cr
u(t,0)=0
\end{cases}
\end{equation}
for any $g\in H(f)$.

By general semigroup theory, for any $u_0\in \hat X$, \eqref{aux-main-eq3-hull} has a unique solution
$u(t,\cdot;u_0,g)$ with $u(0,\cdot;u_0,g)=u_0$. By (H1) and comparison principle for parabolic equations,
we have that for any $u_0\in\hat X^+$, $u(t,\cdot;u_0,g)$ exists and $u(t,\cdot;u_0,g)\in\hat X^+$ for all $t>0$. Moreover,
there is a constant $M(u_0)>0$ such that $u(t,\cdot;u_0,g)\le M(u_0)$
and $|u_x(t, x;u_0,g)|\le M(u_0)$ for $t\ge 0$ and $g\in H(f)$.

Consider \eqref{aux-main-eq2} and
\begin{equation}
\label{aux-main-eq2-hull}
\begin{cases}
u_t=u_{xx}+ug(t,u),\quad -\infty<x<h(t)\cr
u(t,h(t))=0\cr
h^{'}(t)=-\mu u_x(t,h(t))
\end{cases}
\end{equation}
for any $g\in H(f)$.

Note that a solution $u(t,x)$ of \eqref{aux-main-eq3-hull} gives rise to a solution
$(\tilde u(t,x),\tilde h(t))$ of \eqref{aux-main-eq2-hull}, where $\tilde u(t,x)=u(t,\tilde h(t)-x)$ and $\tilde h(t)=\mu \int_0 ^t u_x(s,0)ds$.
Conversely, a solution $(u(t,x),h(t))$ of \eqref{aux-main-eq2-hull} gives rise to a solution $\tilde u(t,x)$ of
\eqref{aux-main-eq3-hull}, where $\tilde u(t,x)=u(t,h(t)-x)$.
Note also that for given $h_0\in \RR$ and $u_0(\cdot)$ satisfying
\begin{equation}
\label{initial-value-1}
u_0(h_0)=0,\,\, u_0(h_0-\cdot)\in\hat X^+,
\end{equation}
\eqref{aux-main-eq2-hull} has a unique solution $(u(t,x;u_0,h_0,g),h(t;u_0,h_0,g))$  with
$(u(0,x;u_0,h_0,g),h(0;u_0,h_0,g))=(u_0(x),h_0)$.

\subsection{Basic properties of diffusive KPP equations in unbounded domains with a free boundary}

In this subsection, we present some basic properties of solutions of
\eqref{aux-main-eq2} and \eqref{aux-main-eq2-hull}.

For given $g\in H(f)$, given $h_{10},h_{20}\in\RR$ and $u_{10}$ and $u_{20}$ satisfying \eqref{initial-value-1}
with $h_0$ being replaced by $h_{10}$ and $h_{20}$, respectively,
assume that $h(t;u_{10},h_{10},g)\le h(t;u_{20},h_{20},g)$ for
$0\le t\le T$. Then $w(t,x):=u(t,x;u_{20},h_{20},g)-u(t,x;u_{10},h_{10},g)$ satisfies
\begin{equation}
\label{aux-main-eq2-1}
w_t=w_{xx}+a(t,x)w,\quad -\infty<x<\eta(t),\,\, 0<t\le T,
\end{equation}
where $\eta(t)=h(t;u_{10},h_{10},g)$ and $a(t,x)=0$ if $u(t,x;u_{20},h_{20},g)=u(t,x;u_{10},h_{10},g)$ and
$$
a(t,x)=\frac{u(t,x;u_{20},h_{20},g)g(t,u(t,x;u_{20},h_{20},g))-u(t,x;u_{10},h_{10},g)g(t,u(t,x;u_{10},h_{10},g))}{u(t,x;u_{20},h_{20},g)-u(t,x;u_{10},h_{10},g)}
$$
if $u(t,x;u_{20},h_{20},g)\not =u(t,x;u_{10},h_{10},g)$.

 \begin{lemma}
 \label{zero-number-lm1}
 Let $\eta(t)$ be a continuous function for $t\in(t_1, t_2)$. If $w(t,x)$ is a continuous function for
 $t\in(t_1, t_2)$ and $x\in(-\infty, \eta(t))$, and satisfies
 \begin{equation*}
 w_t=w_{xx}+a(t,x)w, \quad  x\in(-\infty, \eta(t)),\,\, t\in(t_1, t_2)
 \end{equation*}
 for some bounded continuous function $a(t,x)$ and $w(t,\eta(t))\not =0$, $w(t,x)\not =0$ for
 $x\ll -1$, then for each $t\in(t_1, t_2)$, the number of zero (denoted by $Z(t)$)
 of $w(t, \cdot)$ in $(-\infty, \eta(t)]$ is finite. Moreover $Z(t)$ is nonincreasing in $t$, and if for some
 $s\in(t_1, t_2)$ the function $w(s, \cdot)$ has a degenerate zero $x_0\in(-\infty, \eta(s))$, then
 $Z(s_1)>Z(s_2)$ for all $s_1, s_2$ satisfying $t_1<s_1<s<s_2<t_2$.
 \end{lemma}

\begin{proof}
For any $t_0\in(t_1, t_2)$, by the continuity of $w$ we can find $\epsilon>0, \delta>0$ and $M<0$ such that
$$w(t,x)\neq 0 \  \ for \ \ t\in I_{t_0}:=(t_0-\delta, t_0+\delta),\,\,  x\in\{M\}
\cup [\eta(t_0)-\epsilon, \eta(t)]$$
 Without loss of generality, we may assume that 
 $$
 w(t_0,x)>0\quad {\rm for}\quad -\infty<x\le M.
 $$
 Then 
 $$
 w(t,M)>0\quad {\rm for}\quad t\in (t_0-\delta,t_0+\delta).
 $$
 By comparison principle for parabolic equations, we have
 $$
 w(t,x)>0\quad {\rm for}\quad t\in (t_0,t_0+\delta),\,\, -\infty<x\le M.
 $$
Let  $Z(t;M,\eta(t_0)-\epsilon)$ be the number of zeros of $u(t,\cdot)$ in the interval
$[M, \eta(t_0)-\epsilon]$. We can apply  Theorem D in \cite{ASB} to see that the conclusions for $Z(t;M,\eta(t_0)-\epsilon)$
hold for $t\in I_{t_0}$ and hence $Z(t)=Z(t;M,\eta(t_0)-\epsilon)$ is finite for $t\in[t_0,t_0+\delta)$. This implies that $Z(t)$ is finite for any $t\in (t_1,t_2)$.
Moreover, 
$$Z(t)\ge Z(t;M,\eta(t_0)-\epsilon)\ge Z(t_0;M,\eta(t_0)-\epsilon)=Z(t_0)\quad {\rm for}\quad t\in(t_0-\delta,t_0),
$$
$$Z(t)= Z(t;M,\eta(t_0)-\epsilon)\le Z(t_0;M,\eta(t_0)-\epsilon)=Z(t_0)\quad {\rm for}\quad t\in (t_0,t_0+\delta),
$$
and if $w(t_0, \cdot)$ has a degenerate zero $x_0\in(-\infty, \eta(t_0))$, then
 $Z(s_1)>Z(s_2)$ for all $s_1, s_2$ satisfying $t_1<s_1<t_0<s_2<t_2$.

\end{proof}

\begin{lemma}
\label{zero-number-lm2}
For given $g\in H(f)$, $h_{10}, h_{20}\in\RR$, and  $u_{10}$, $u_{20}$ satisfying \eqref{initial-value-1} with $h_0$ being replaced
by $h_{10}$ and $h_{20}$, respectively. If $u^{'}_{20}(x_2)<u^{'}_{10}(x_1)$
for any $x_1,x_2$ such that $u_{20}(x_2)=u_{10}(x_1)$, then
$$u(s,x+h(s;u_{20},h_{20},g);u_{20},h_{20},g)\ge u(s,x+h(s;u_{10},h_{10},g);u_{10},h_{10},g)
$$
for $x\le 0$ and $s\ge 0$.
\end{lemma}

\begin{proof}
Fix any $s> 0$.
Let $\tilde u_1(t,x)=u(t,x+h(s;u_{10},h_{10},g);u_{10},h_{10},g)$ and
$\tilde u_2(t,x)=u(t,x+h(s;u_{20},h_{20},g);u_{20},h_{20},g)$.
Then
$$
\tilde u_1(t,x)=u(t,x;u_{10}(\cdot+h(s;u_{10},h_{10},g)), h_{10}-h(s;u_{10},h_{10},g),g)
$$
and
$$
\tilde u_2(t,x)=u(t,x;u_{20}(\cdot+h(s;u_{20},h_{20},g)), h_{20}-h(s;u_{20},h_{20},g),g).
$$
Note that
$$
\tilde u_1(s,0)=\tilde u_2(s,0).
$$
We must have
$$
h_{20}-h(s;u_{20},h_{20},g)< h_{10}-h(s;u_{10},h_{10},g)
$$
and there is a unique $\xi(0)<h_{20}-h(s;u_{20},h_{20},g)$ such that
$$
\tilde u_2(0,x) \begin{cases}
>\tilde u_1(0,x)\,\, {\rm for}\,\, x<\xi(0)\cr
<\tilde u_1(0,x)\,\, {\rm for}\,\, \xi(0)<x<h_{20}-h(s;u_{20},h_{20},g).
\end{cases}
$$
Then by the zero number property (see Lemma \ref{zero-number-lm1}),
$$
\tilde u_2(s,x)> \tilde u_1(s,x),\quad -\infty<x<0.
$$
The lemma then follows.
\end{proof}

Let $H(x)$ be a $C^2((-\infty,0])$ function with $H^{'}(x)\le 0$, $H(0)=0$, $H(x)=1$ for $x\le -1$.
For given $g\in H(f)$, let $u_{0,g}(x)$ and $u_{n,g}(x)$ be defined by
$$
u_{0,g}(x)=\begin{cases} u_g(0),\quad &x<0\cr
0,\quad &x=0
\end{cases}
$$
and
$$
u_{n,g}(x)=H(nx)u_{0,g}(x).
$$
Then
$$
u_{n,g}(x)\ge u_{m,g}(x),\quad \forall n\ge m\,\, x\le 0
$$
and
$$
u_{n,g}(x)\to u_{0,g}(x),\quad \forall\,\, x\le 0
$$
as $n\to\infty$. By Proposition \ref{comparison-principle1},  for any $h_0\in\RR$ and $n\ge m$, we have
$$
h(t;u_{n,g}(\cdot-h_0),h_0,g)\ge h(t;u_{m,g}(\cdot-h_0),h_0,g)\quad \forall\,\, t>0
$$
and
$$
u(t,x;u_{n,g}(\cdot-h_0),h_0,g)\ge u(t,x;u_{m,g}(\cdot-h_0),h_0,g)\quad \forall\,\, x\le h(t;u_{m,g}(\cdot-h_0),h_0,g), \,\, t\ge 0.
$$
Let
$$
h(t;u_{0,g}(\cdot-h_0),h_0,g)=\lim_{n\to\infty} h(t;u_{n,g}(\cdot-h_0),h_0,g)\quad \forall\,\, t\ge 0
$$
and
$$
u(t,x;u_{0,g}(\cdot-h_0),h_0,g)=\begin{cases}
\lim_{n\to\infty}u(t,x;u_{n,g}(\cdot-h_0),h_0,g),\,\,& x<h(t;u_{0,g}(\cdot-h_0),h_0,g)\cr
0\quad &x=h(t;u_{0,g}(\cdot-h_0),h_0,g).
\end{cases}
$$
Then we have that $(u(t,x;u_{0,g}(\cdot-h_0),h_0,g),h(t;u_{0,g}(\cdot-h_0),h_0,g))$ is a solution of
\eqref{aux-main-eq2-hull} for $t>0$ and
$$
(u(0,x;u_{0,g}(\cdot-h_0),h_0,g),h(0;u_{0,g}(\cdot-h_0),h_0,g))=(u_{0,g}(x-h_0),h_0)\quad \forall\,\, x\le h_0.
$$

\begin{lemma}
\label{zero-number-lm3}
For any given $g\in H(f)$,  $h_{10},h_{20}\in \RR$ and $u_{20}$ satisfying \eqref{initial-value-1} with
$h_0=h_{20}$ and $u_{20}(x)<u_g(0)$ for all $x\le h_{20}$, there holds
$$
u(s,x+h(s;u_{0,g}(\cdot-h_{10}),h_{10},g);u_{0,g}(\cdot-h_{10}),h_{10},g)\ge u(s,x+h(s;u_{20},h_{20},g);u_{20},h_{20},g)
$$
for all $x\le 0$ and $s\ge 0$.
\end{lemma}

\begin{proof}
First, we note that for any $n$ large enough, $u^{'}_{n,g}(x_1)<u^{'}_{20}(x_2)$ for any $x_1,x_2$ satisfying
that $u_{n,g}(x_1)=u_{20}(x_2)$. Then by Lemma \ref{zero-number-lm2},
$$
u(s,x+h(s;u_{n,g}(\cdot-h_{10}),h_{10},g);u_{n,g}(\cdot-h_{10}),h_{10},g)\ge u(s,x+h(s;u_{20},h_{20},g);u_{20},h_{20},g)
$$
for all $x\le 0$, $s\ge 0$, and $n\gg 1$. Letting $n\to\infty$, we have
$$
u(s,x+h(s;u_{0,g}(\cdot-h_{10}),h_{10},g);u_{0,g}(\cdot-h_{10}),h_{10},g)\ge u(s,x+h(s;u_{20},h_{20},g);u_{20},h_{20},g)
$$
for all $x\le 0$ and $s\ge 0$. The lemma is thus proved.
\end{proof}

\subsection{Basic properties of diffusive KPP equations in fixed unbounded domains}

In this section, we presentation some basic properties of solutions of \eqref{aux-main-eq3} and \eqref{aux-main-eq3-hull}.

First of all, by the relation between the solutions of
\eqref{aux-main-eq3-hull} and \eqref{aux-main-eq2-hull}, we have

\begin{lemma}
\label{zero-number-lm-cor}
\begin{itemize}
\item[(1)]
For given $u_{01},u_{02}\in \hat X^+$, if $u^{'}_{01}(x)\ge 0$, $u^{'}_{02}(x)\ge 0$,
and $u^{'}_{02}(x_2)>u_{01}^{'}(x_1)$ for any $x_1,x_2\ge 0$ satisfying that $u_{01}(x_1)=u_{02}(x_2)$,
then
$$
u(t,x;u_{01},g)\le u(t,x;u_{02},g) \quad \forall\,\, x\ge 0,\, \, t\ge 0.
$$

\item[(2)] For any $u_0\in \hat X^+$ with $u_0(x)<u_g(0)$, there holds
$$
u(t,x;\tilde u_{0,g},g)\ge u(t,x;u_0,g)\quad \forall\,\, x\ge 0,\, \, t\ge 0,
$$
where $\tilde u_{0,g}(x)=u_{0,g}(-x)$ and $u(t,x;\tilde u_{0,g},g)=u(t,h(t;u_{0,g},0,g)-x;u_{0,g},0,g)$.
\end{itemize}
\end{lemma}

\begin{proof}
(1) It follows directly from Lemma \ref{zero-number-lm2}.

(2) It follows from  Lemma \ref{zero-number-lm3}.
\end{proof}

\begin{lemma}
\label{monotonicity-for-semi-wave-lm}
Consider \eqref{aux-main-eq3-hull}.
For any $u_0\in\hat X^+$ with   $u^{'}_0(x)\ge 0$ and $u^{'}_0(0)>0$, then
$u_x(t,x;u_0,g)>0$ for all
$t>0$, $x\ge 0$,  and $g\in H(f)$.
\end{lemma}

\begin{proof}
First of all, it is easily known that $u_0(x)>0$ for $x>0$.
By comparison principle for parabolic equations,
$u(t,x;u_0,g)\ge 0$ for all $t\ge 0$, $x\ge 0$ and $g\in H(f)$. Hence
$$
u_x(t,0;u_0,g)\ge 0\quad \forall\,\, t\ge 0, \,\, x\ge 0,\,\,{\rm and}\,\, g\in H(f).
$$
 Note that $v(t,x)=u_x(t,x;u_0,g)$ is the solution
of
$$
\begin{cases}
v_t=v_{xx}-\mu u_x(t,0;u_0,g)v_x(t,x)+[g(t,u(t,x;u_0,g))\cr
\qquad\qquad +u(t,x;u_0,g)g_u(t,u(t,x;u_0,g))]v(t,x),
\quad 0<x<\infty\cr
v(t,0)\ge 0\cr
v(0,x)=u^{'}_0(x)\ge 0.
\end{cases}
$$
Then by comparison principle for parabolic equations again,
$$u_x(t,x;u_0,g)\ge 0\quad \forall\,\,
t>0,\,\, x\ge 0, \,\,{\rm and}\,\, g\in H(f).
$$

Next, by Hopf Lemma and strong maximum principle for parabolic equations, we have
$$
u_x(t,x;u_0,g)>0\quad \forall\,\, t>0,\,\, x\ge 0\,\, {\rm and}\,\, g\in H(f).
$$
\end{proof}

For given $u_1,u_2\in\hat X^{++}$ with $u_1(\cdot)\le u_2(\cdot)$, we define a metric, $\rho(u_1,u_2)$, between $u_1$ and $u_2$
as follows,
$$
\rho(u_1,u_2)=\inf\{\ln\alpha\,|\, \alpha\ge 1,\,\, u_2(\cdot)\le \alpha u_1(\cdot)\}.
$$
For given $u_1,u_2\in\hat X^{++}$ with $u_i^{'}(0)>0$ and  $u_i^{'}(x)\ge 0$, by Lemma \ref{monotonicity-for-semi-wave-lm},
$u(t,\cdot;u_i,g)\in \hat X^{++}$ for $t>0$ and $g\in H(f)$.

\begin{lemma}
\label{part-metric-lm}
 Consider \eqref{aux-main-eq3-hull}.
For any $u_0,v_0\in\hat X^{++}$ with $u_0(\cdot)\not =v_0(\cdot)$, if
$u(t,\cdot;u_0,g)$, $u(t,\cdot;v_0,g)\in \hat X^{++}$, and $u(t,\cdot;u_0,g)\le u(t,\cdot;v_0,g)$
for all $t>0$,
then
$$
\rho(u(t_2,\cdot;u_0,g),u(t_2,\cdot;v_0,g))\le \rho(u(t_1,\cdot;u_0,g),u(t_1,\cdot;v_0,g))
$$
for all $0\le t_1< t_2$ and $g\in H(f)$. Moreover, if $\lim_{x\to\infty}u_0(x)=\lim_{x\to\infty}v_0(x)$, then
$$
\rho(u(t_2,\cdot;u_0,g),u(t_2,\cdot;v_0,g))< \rho(u(t_1,\cdot;u_0,g),u(t_1,\cdot;v_0,g)).
$$
\end{lemma}

\begin{proof}
First,
for any $u_0, v_0\in \hat X^{++}$ with $u_0(\cdot)\le v_0(\cdot)$, $u_0(\cdot)\not = v_0(\cdot)$,
 there is $\alpha^*>1$ such that  $\rho(u_0,v_0)=\ln \alpha^{*}$
and $ v_0\leq \alpha^{*} u_0$.
Let
 $$w(t, x)=\alpha^{*}u(t, x; u_0,g)$$
We then have
\begin{align*}
w_{t}(t,x)&=w_{xx}(t,x)-\mu u_x(t,0;u_0,g)w_x(t,x)+w(t,x)g(t,u(t,x;u_0,g))\\
& =
w_{xx}(t,x)-\mu u_x(t,0;u_0,g)w_x(t,x)+w(t,x)g(t,w(t,x))\\
&\qquad +w(t,x)g(t,u(t,x;u_0,g))-w(t,x)g(t,w(t,x))\\
&> w_{xx}(t,x)-\mu u_x(t,0;u_0,g)w_x(t,x)+w(t,x)g(t,w(t,x))\\
&\geq w_{xx}(t,x)-\mu u_x(t,0;v_0,g)w_x(t,x)+w(t,x)g(t,w(t,x))
\quad \ for \ all  \  \ t>0,\quad x\in \RR^+,
\end{align*}
and
$$
w(t,0)=0,\quad \ for \ all  \  \  t>0.
$$
By comparison principle for parabolic equations, we have
$$
u(t, x;v_0,g)\le  \alpha^{*}u(t,x; u_0,g)
$$
for $t>0$ and $x>0$.
Therefore,
$$
\rho(u(t,\cdot;u_0,g),u(t,\cdot;v_0,g))\le \rho(u_0,v_0)\quad \ for \  all  \  \  t\ge 0
$$
and then
$$
\rho(u(t_2,\cdot;u_0,g),u(t_2,\cdot;v_0,g))\le \rho(u(t_1,\cdot;u_0,g),u(t_1,\cdot;v_0,g))
\quad \ for \ all \,\, 0\le t_1< t_2.
$$

Assume that $u_\infty=\lim_{x\to\infty}u_0(x)=\lim_{x\to\infty}v_0(x)$. Then for any $t>0$,
\begin{equation}
\label{aux-part-metric-eq0}
\lim_{x\to\infty}u(t,x;u_0,g)=\lim_{x\to\infty}u(t,x;v_0,g)=u(t;u_\infty,g),
\end{equation}
where $u(t;u_\infty,g)$ is the solution of \eqref{ode-eq1-hull} with $u(0;u_\infty,g)=u_\infty$.
Since $\alpha^*>1$, $u_\infty\not =\alpha^* u_\infty$. Hence $v_0\neq \alpha^* u_0$. By Hopf Lemma,
\begin{equation}
\label{aux-part-metric-eq1}
u_x(t,0;v_0,g)<\alpha^* u_x(t,0;u_0,g).
\end{equation}
By \eqref{aux-part-metric-eq0},
\begin{equation}
\label{aux-part-metric-eq3}
\lim_{x\to\infty}u(t,x;v_0,g)=u(t;u_\infty,g)<\alpha^* u(t;u_\infty,g)=\alpha^*\lim_{x\to\infty}u(t,x; u_0,g).
\end{equation}
By \eqref{aux-part-metric-eq1}-\eqref{aux-part-metric-eq3}, there is $0<\beta<1$ such that
$$
u(t,x;v_0,g)\le \beta \alpha^* u(t,x;u_0,g).
$$
It then follows that
$$
\rho(u(t,\cdot;u_0,g),u(t,\cdot;v_0,g))<\rho(u_0,v_0)
$$ and then for any $0\le t_1<t_2$,
$$
\rho(u(t_2,\cdot;u_0,g),u(t_2,\cdot;v_0,g))<\rho(u(t_1,\cdot;u_0,g),u(t_1,\cdot;v_0,g)).
$$
\end{proof}

\section{Semi-Wave Solutions and Proof of Theorem \ref{main-thm1}}

In this section, we  investigate the semi-wave solutions of
 \eqref{aux-main-eq2} and prove Theorem \ref{main-thm1}.

We first prove some lemmas.

\begin{lemma}
\label{semi-wave-lm1}
Let $g\in H(f)$ be given. There is $u_0\in \hat X^{++}$ such that
$u_0^{'}(x)\ge 0$ for $x\ge 0$ and $\inf_{t\ge 0}u_x(t,0;u_0,g)>0$.
\end{lemma}

\begin{proof}
Let $\epsilon>0$ be given such that $\tilde\lambda(a,\epsilon\mu,l)>0$ for $l\gg 1$.

First of all, there is $K>0$ such that
$$
0\le u_x(t,x;u_0,g)\le K,\,\, |u_{xx}(t,x;u_0,g)|\le K
$$
for any $u_0\in \hat X^{++}$ with $u_0^{'}(x)\ge 0$ for $x\ge 0$, $u_0^{'}(x)=0$ for $x\ge 1$, and
$\|u_0\|_{\hat X}\ll \min\{\frac{\epsilon}{2},\frac{\epsilon^2}{4K}\}$.  Fix such a
$u_0$ with $u_0(\cdot)\not \equiv 0$.

Observe that $u_x(t,0;u_0,g)<\epsilon$ for $0<t\ll 1$. Let
$$
t_1=\sup\{\tau \,|\, u_x(t,0;u_0,g)<\epsilon,\,\, \forall\,\, t\in [0,\tau)\}.
$$
Then
$u_x(t,0;u_0,g)<\epsilon$ for $t\in (0,t_1)$ and $u_x(t_1,0;u_0,g)=\epsilon$ in the case $t_1<\infty$.
By comparison principle for parabolic equations,
\begin{equation}
\label{aux-semi-wave-eq2}
u(t,x;u_0,g)\ge u_\epsilon(t,x;u_0,g)\quad {\rm for}\quad 0\le t < t_1,
\end{equation}
where $u_\epsilon(t,x;u_0,g)$ is the solution of \eqref{aux-semi-wave-eq1} with
$u_\epsilon (0,x;u_0,g)=u_0(x)$.

 Next, if $t_1=\infty$, by Lemma \ref{semi-wave-lm1-1},
 the lemma is proved. Otherwise, note that
 $$
 u_x(t,x;u_0,g)=u_x(t,0;u_0,g)+\int_{0}^x u_{xx}(t,y;u_0,g)dy\ge u_x(t,0;u_0,g)-K x.
 $$
 Hence for $0<x<\frac{\epsilon}{2K}$,
 $$
 u_x(t_1,x;u_0,g)\ge \frac{\epsilon}{2}
 $$
 and
 $$
 u(t_1,\frac{\epsilon}{2K};u_0,g)\ge \frac{\epsilon^2}{4 K}.
 $$
 We then have that
 $$
 u_x(t_1,x_1;u_0,g)>u_0^{'}(x_0)
 $$
 for any $x_0,x_1\ge 0$ such that $u(t_1,x_1;u_0,g)=u_0(x_0)$.  By Lemma  \ref{zero-number-lm-cor},
 we have
 \begin{equation*}
 u(t+t_1,x;u_0,g)\ge u(t,x;u_0,g\cdot t_1)\quad {\rm for}\quad t\ge 0.
 \end{equation*}

 Similarly, let
 $$
 t_2=\sup\{\tau\,|\, u_x(t,0;u_0,g\cdot t_1)<\epsilon\quad \forall\,\, t\in [0,\tau)\}.
 $$
 Then\begin{equation}
 \label{aux-semi-wave-eq3}
u(t+t_1,x;u_0,g)\ge u(t,x;u_0,g\cdot t_1)\ge u_\epsilon(t,x;u_0,g\cdot t_1)\quad {\rm for}\quad 0\le t < t_2
 \end{equation}
  and in the case $t_2<\infty$,
 \begin{equation*}
 u(t+t_1+t_2,x;u_0,g)\ge u(t+t_2,x;u_0,g\cdot t_1)\ge u(t,x;u_0,g\cdot(t_1+t_2))\quad {\rm for}\quad t\ge 0.
 \end{equation*}

 Repeating the above process, if $t_{n-1}<\infty$,  let
 $$
 t_n=\sup\{\tau\,|\, u_x(t,0;u_0,g\cdot(t_1+\cdots+t_{n-1}))<\epsilon\quad {\rm for} \,\, t\in [0,\tau)\},
  $$
  $n=1,2,\cdots$.
 Then
 \begin{align}
 \label{aux-semi-wave-eq4}
 u(t+t_1+\cdots+t_{n-1},x;u_0,g)&\ge
u(t,x;u_0,g\cdot(t_1+\cdots+t_{n-1}))\nonumber \\
&\ge u_\epsilon (t,x;u_0,g\cdot(t_1+\cdots +t_{n-1}))
\quad {\rm for}\quad 0\le t < t_n
\end{align}
and in the case $t_n<\infty$,
  \begin{equation*}
   u(t+t_1+\cdots+t_n,x;u_0,g)\ge u(t,x;u_0,g\cdot (t_1+\cdots +t_n))\quad {\rm for}\quad t\ge 0.
  \end{equation*}

  It is not difficult to see that
  $\inf_{n\ge 1}t_n>0$.
  Then by \eqref{aux-semi-wave-eq2}-\eqref{aux-semi-wave-eq4}  and Lemma \ref{semi-wave-lm1-1} again,
  $
  \inf_{t\ge 0}u_x(t,0;u_0,g)>0$.
\end{proof}

\begin{lemma}
\label{semi-wave-lm2}
For any $\epsilon>0$, there are $T^*>0$ and $x^*>0$ such that
$$
|u(t,x;u_0,g)-u_g(t)|<\epsilon
$$
for $t\ge T^*$ and $x\ge x^*$, where $u_0$ is as in Lemma \ref{semi-wave-lm1}
with $u_0(x)\le u_g(0)$.
\end{lemma}

\begin{proof}
First, note  that $u_{\inf}:=\inf_{t\ge 0,x\ge 1}u(t,x;u_0,g)>0$ and
$u_\infty:=\lim_{x\to\infty}u_0(x)>0$.
We then have
$$
\lim_{x\to\infty}u(t,x;u_0,g)=u(t;u_\infty,g)
$$
where $u(t;u_\infty,g)$ is the solution of \eqref{ode-eq1-hull} with
$u(0;u_\infty,g)=u_\infty$. Also note that, by Lemma \ref{ode-lm},
 for any $\epsilon>0$, any $\tilde g\in H(f)$, there is $T^*>0$ such that
 $$
 |u(t;u_{\inf},\tilde g)-u_{\tilde g}(t)|<\epsilon/4
 $$
 for $t\ge T^*$.

 We claim that there is $x^*\ge 1$ such that
 $$
 |u(t,x;u_0,g)-u_g(t)|<\epsilon
 $$
 for $t\ge T^*$ and $x\ge x^*$. In fact, assume this is not true,
 then for any $n\ge 1$, there are $x_n\ge n$ and $t_n\ge T^*$ such that
 $$
 |u(t_n,x_n;u_0,g)-u_{g}(t_n)|\ge \epsilon.
 $$
 Let
 $$
 u_n(t,x)=u(t-T^*+t_n,x+x_n;u_0,g).
 $$
 Without loss of generality, assume that
 $$
 g\cdot (t_n-T^*)\to \tilde g, \quad u_n(t,x)\to \tilde u(t,x),\quad u_x(t-T^*+t_n,0;u_0,g)\to \tilde \xi(t)
 $$
 as $n\to\infty$. Then $\inf_{t\ge 0, x\in\RR} \tilde u(t,x)\ge u_{\inf}$ and
 $\tilde u(t,x)$ satisfies
 of
 \begin{equation}
 \label{new-aux-eq1}
  u_t=u_{xx}-\mu \tilde\xi(t) u_x+u  \tilde g(t, u),\quad x\in\RR,\,\, t\ge 0.
 \end{equation}
 Note that $u(t;u_{\inf},\tilde g)$ is also the solution of
 \eqref{new-aux-eq1} with $u(0;u_{\inf},\tilde g)=u_{\inf}$. By comparison principle
 for parabolic equations, we have
 $$
 \tilde u(t,x)\ge u(t;u_{\inf},\tilde g)> u_{\tilde g}(t)-\epsilon/4
 $$
 for $t\ge T^*$ and any $x\in\RR$. Then for $n\gg 1$,
 \begin{align*}
 u(t_n,x_n;u_0,g)&=u_n(T^*,0)&\\
 &\ge \tilde u(T^*,0)-\epsilon/4\\
 &>u_{\tilde g}(T^*)-\epsilon/2\\
 &>u_{g\cdot (t_n-T^*)}(T^*)-\epsilon\\
 &=u_g(t_n)-\epsilon.
 \end{align*}
 Note that
 $$
 u(t,x;u_0,g)\le u_g(t)\quad \forall\,\, t\ge 0,\,\, x\ge 0.
 $$
 We then have
 $$
 |u(t_n,x_n;u_0,g)-u_g(t_n)|<\epsilon.
 $$
 This is a contradiction. The claim is then true and the lemma follows.
\end{proof}

\begin{corollary}
\label{semi-wave-cor}
For any $\tilde g\in H(f)$, let $t_n\to\infty$ be such that $g\cdot t_n\to \tilde g$ and
$u(t_n,\cdot;u_0,g)\to \tilde u_{\tilde g}$. Then
$u^{**}(t,x;\tilde g)=u(t,x;\tilde u_{\tilde g},\tilde g)$ is an entire positive solution of \eqref{aux-main-eq3-hull} with
$g$ being replaced by $\tilde g$ and $\lim_{x\to\infty}u^{**}(t,x;\tilde g)=u_{\tilde g}(t)$ uniformly in $t\in\RR$.
\end{corollary}

\begin{proof}
It follows from Lemmas \ref{semi-wave-lm1} and \ref{semi-wave-lm2} directly.
\end{proof}

Let $\tilde u_{0,g\cdot t}(x)=u_{0,g\cdot t}(-x)$ for any $x\in\RR^+$.
Observe that, for any given $g\in H(f)$ and $T_2>T_1>t>0$, we have
$$
u(T_2+t,x;\tilde u_{0,g\cdot (-T_2)},g\cdot(-T_2))=u(T_1+t,x;u(T_2-T_1,\cdot;\tilde u_{0,g\cdot(-T_2)},g\cdot(-T_2)),g\cdot(-T_1)).
$$
Then by Lemma \ref{zero-number-lm-cor},
$$
u(T_2+t,x;\tilde u_{0,g\cdot(-T_2)},g\cdot(-T_2))\le u(T_1+t,x;\tilde u_{0,g\cdot (-T_1)},g\cdot (-T_1)).
$$
Let
$$
U^*(t,x;g)=\lim_{T\to\infty}u(t+T,x;\tilde u_{0,g\cdot(-T)},g\cdot(-T)).
$$
Then $U^*(t,x;g)$ is an entire solution of \eqref{aux-main-eq3-hull},

\begin{lemma}
\label{semi-wave-lm3}
For any entire positive solution $v(t,x)$ of \eqref{aux-main-eq3-hull} with $v(t,x)<u_g(t)$,
$$
v(t,x)\le U^*(t,x;g).
$$
Moreover,  $\lim_{x\to\infty}U^*(t,x;g)=u_g(t)$ uniformly in $t\in\RR$ and $g\in H(f)$.
\end{lemma}

\begin{proof}
First, let $v(t,x)$ be an entire positive solution of \eqref{aux-main-eq3-hull}. By Lemma \ref{zero-number-lm-cor},
$$
v(t,x)=u(t+T,x;v(-T,\cdot),g\cdot(-T))\le u(t+T,x;\tilde u_{0,g\cdot(-T)},g\cdot (-T))
$$
for any $t\in\RR$, $t+T>0$, and $x\ge 0$. Letting $T\to\infty$, we have
$$
v(t,x)\le U^*(t,x;g)\quad \forall\,\, x\ge 0.
$$

Next, let $u^{**}(t,x;g)$ be the entire solution in Corollary \ref{semi-wave-cor}. By the above arguments,
$$
u^{**}(t,x;g)\le U^*(t,x;g).
$$
Note that $U^*(t,x;g)\le u_g(t)$. Then
$$
0\le u_g(t)-U^*(t,x;g)\le u_g(t)-u^{**}(t,x;g)\to 0
$$ as $x\to\infty$ uniformly in $t\in\RR$ and $g\in H(f)$.
\end{proof}

\begin{proof} [Proof of Theorem \ref{main-thm1}]
Let $\tilde u^{**}(t,x)=U^*(t,x;f)$, we only need to prove $U^*(t,x;g)$ satisfies the properties
in Theorem \ref{main-thm1} for any $g\in H(f)$. Theorem \ref{main-thm1} then follows.

(1) It suffices to prove that $U^*(t,x;g)$ is almost periodic in $t$.

Note that $U^*(t,x;g)=U^*(0,x;g\cdot t)$ for any $t\in\RR$ and $g\in H(f)$.
We claim that $g\in H(f)\mapsto U^*(0,\cdot;g)\in \hat X^{++}$ is continuous.
Assume that there is $g_n\in H(f)$ such that $g_n\to g^*$ and
$$
U^*(0,\cdot;g_n)\to \tilde U^*(\cdot)\not = U^*(0,\cdot;g^*).
$$
Then $u(t,x;\tilde U^*,g^*)$ is an entire solution and
$$
U^*(t,x;g^*)\ge u(t,x;\tilde U^*,g^*).
$$
Note that $\rho(U^*(t,\cdot;g^*),u(t,\cdot;\tilde U^*,g^*))$ is nonincreasing in $t$.
Let
$$
\rho_{-\infty}=\lim_{t\to -\infty} \rho(U^*(t,\cdot;g^*),u(t,\cdot;\tilde U^*,g^*)).
$$
Then $\rho_{-\infty}\not =0$.
Take a sequence $s_n\to -\infty$ such that $g^*\cdot s_n\to g^{**}$,
$U^*(s_n,\cdot;g^*)\to U^{**}(\cdot)$, and $u(s_n,\cdot;\tilde U^*,g^*)\to \tilde U^{**}(\cdot)$.
Then
$$
u(t,x;U^{**},g^{**})=\lim_{n\to\infty} U^*(t+s_n,x;g^*)
$$
and
$$
u(t,x;\tilde U^{**},g^{**})=\lim_{n\to\infty} u(t+s_n,x;\tilde U^*,g^*).
$$
Hence
$$
u(t,x;U^{**},g^{**})\ge u(t,x;\tilde U^{**},g^{**}).
$$
Hence $\rho(u(t,\cdot;U^{**},g^{**}),u(t,\cdot;\tilde U^{**},g^{**}))$ is well defined and
$$
\rho(u(t,\cdot;U^{**},g^{**}),u(t,\cdot;\tilde U^{**},g^{**}))=\rho_{-\infty}
$$ for all $t\in\RR$. This implies that $u(t,\cdot;U^{**},g^{**})=u(t,\cdot;\tilde U^{**},g^{**})$
and $\rho_{-\infty}=0$, which is a contradiction. Therefore,
$g\in H(f)\mapsto U^*(0,\cdot;g)\in\hat X^{++}$ is continuous. We then have $U^*(t,\cdot;g)=U^*(0,\cdot;g\cdot t)$ is almost periodic in $t$.

(2) Suppose that $u^{**}(t,x;g)$ is also an almost periodic positive solution
of \eqref{aux-main-eq3-hull} and
$$\lim_{x\to\infty}u^{**}(t,x;g)=u_g(t)$$
 uniformly in $t\in\RR$.
Then  by Lemma \ref{semi-wave-lm3},
$$
U^*(t,x;g)\ge u^{**}(t,x;g).
$$
By the almost periodicity, there is $t_n\to\infty$ such that $g\cdot t_n\to g$ and
$$
U^*(t_n,x;g)\to U^*(0,x;g),\quad u^{**}(t_n,x;g)\to u^{**}(0,x;g)
$$
as $n\to\infty$ uniformly in $x\ge 0$. It then follows that
$$
\rho(u^{**}(t,\cdot;g),U^*(t,\cdot;g))={\rm constant}
$$
and then we must have $u^{**}(t,x;g)\equiv U^{*}(t,x;g)$.

(3)
For any bounded positive solution $u(t,x)$ of \eqref{aux-main-eq3-hull} with $\liminf_{x\to\infty}\inf_{t\ge 0}u(t,x)>0$,  suppose that
$$
\lim_{t\to\infty}[U^*(t,x;g)-u(t,x)]\neq0
$$
then there exist $t_n\to\infty$, $u^*\in\hat X^{++}$, such that
$g\cdot t_n\to g^*$, $U^*(t_n,x;g)\to U^*(0,x;g^*)$, $u(t_n,x)\to u^*(x)$
and $U^*(0,\cdot;g^*)\neq u^*(\cdot)$.
Note that $U^*(t,\cdot;g^*)$ and $u(t,\cdot;u^*,g^*)$ exists for all $t\in\RR$,
and by Lemma \ref{semi-wave-lm3} we have
$$
U^*(t,x;g^*)\ge u(t,x;u^*,g^*)
$$
Then $\rho(U^*(t,\cdot;g^*),u(t,\cdot;u^*,g^*))$ is well defined and decreases as $t$ increases.
Let
$$
\rho_{-\infty}=\lim_{t\to-\infty}\rho(U^*(t,\cdot;g^*),u(t,\cdot;u^*,g^*))
$$
Then $\rho_{-\infty}>0$. By the same arguments in (1), we can get $\rho_{-\infty}=0$, which is a
contradiction. Therefore
$$
\lim_{t\to\infty}[U^*(t,x;g)-u(t,x)]=0
$$
uniformly in $x\ge0$.

\end{proof}

\section{Spreading Speeds in Diffusive KPP Equations with Free Boundary and Proof of Theorem \ref{main-thm2}}

In this section, we consider spreading speeds in spatially homogeneous diffusive KPP equations
with free boundary and prove Theorem \ref{main-thm2}.

\begin{proof}[Proof of Theorem \ref{main-thm2}]
 We divide the proof into four steps. We put $u(t,x)=u(t,x;u_0,h_0)$ and
 $h(t)=h(t;u_0,h_0)$ if no confusion occurs.

\textbf{Step 1.} We prove that
the unique positive almost periodic solution $V^*(t)$ of
the problem \eqref{ode-eq1} satisfies
$$\underline{V}_\epsilon(t)\leq V^*(t)\leq\bar{V}_\epsilon(t)
$$
where $\bar{V}_\epsilon(t)$ and $\underline{V}_\epsilon(t)$ are,
respectively, the unique positive almost periodic solution of
\begin{equation}
\label{ordinary-equa1}
V_t=V(f(t,V)+\epsilon)
\end{equation}
and
\begin{equation}
\label{ordinary-equa2}
V_t=V(f(t,V)-\epsilon),
\end{equation}
and $0<\epsilon\ll 1$.

Obviously, $\bar{V}_\epsilon$ and $\underline{V}_\epsilon$
are, respectively, the supersolution and subsolution of
\eqref{ode-eq1}. Hence, by the comparison principle and uniqueness and stability of almost periodic positive solutions
of \eqref{ode-eq1},
we have
$$\underline{V}_\epsilon(t)\leq V^*(t)\leq\bar{V}_\epsilon(t).
$$

Furthermore, for any $0<\epsilon\ll 1$, consider the following two problems
\begin{equation}
\label{semi-wave-eq4}
\begin{cases}
v_t=v_{xx}-\mu v_x(t,0) v_x(t,x)+ v (f(t,v)+\epsilon),\quad 0<x<\infty \cr
v(t,0)=0 \quad
\end{cases}
\end{equation}
and
\begin{equation}
\label{semi-wave-eq5}
\begin{cases}
z_t=z_{xx}-\mu z_x(t,0) z_x(t,x)+ z(f(t,z)-\epsilon),\quad 0<x<\infty \cr
z(t,0)=0. \quad
\end{cases}
\end{equation}

Using the same arguments as in Theorem \ref{main-thm1}, we know that
there exist the unique positive almost periodic solution
$v_\epsilon(t,x)$ of \eqref{semi-wave-eq4} and
$z_\epsilon(t,x)$ of \eqref{semi-wave-eq5} such that
$$\lim_{x\to\infty}v_\epsilon(t,x)=\bar{V}_\epsilon(t) $$
and
$$\lim_{x\to\infty}z_\epsilon(t,x)=\underline{V}_\epsilon(t) $$
uniformly in $t\in\RR$.
Let $\epsilon\to 0$, we can get  $\bar{V}_\epsilon(t)$ and
$\underline{V}_\epsilon(t)$ converge to $V^*(t)$ uniformly in $t\in\RR$.

\textbf{Step 2.}
We prove
$${\overline{\lim}}_{t\to\infty}\frac{h(t)}{t}\leq c^*.$$

By Proposition \ref{main-results-of-part1},
\begin{equation}
\begin{split}
\label{6-6}
\lim_{t\to\infty}u(t,x)-V^*(t)=0
\ \ locally \ \ uniformly \ in \ x\ge 0.
\end{split}
\end{equation}
Since $h_\infty=\infty$, there exists a $T>0$ such that
$$h(T)>l^* \ and \ u(t+T,l^*)\leq\bar{V}_\epsilon(t+T)
\ \ for \ all \ t\geq0.$$

Let
$$\tilde{u}(t,x)=u(t+T,x+l^*) \ and \ \tilde{h}(t)=h(t+T)-l^*.$$
We obtain
\begin{equation*}
\begin{cases}
\tilde{u}_t=\tilde{u}_{xx}+\tilde{u}f(t+T,\tilde{u})
\quad &t>0, 0<x<\tilde{h}(t) \cr
\tilde{u}(t,0)=u(t+T,l^*),
\tilde{u}(t,\tilde{h}(t)=0 \quad &t>0 \cr
\tilde{h}^{'}(t)=-\mu\tilde{u}_x(t,\tilde{h}(t))
\quad &t>0 \cr
\tilde{u}(0,x)=u(T,x+l^*) \quad &0<x<\tilde{h}(0).
\end{cases}
\end{equation*}

Let $u^*(t)$ be the unique positive solution of the problem
\begin{equation*}
\begin{cases}
u^*_t=u^*(f(t,u^*)+\epsilon) \quad  \ t>T \cr
u^*(T)=\max\{\bar{V}_\epsilon, \|\tilde{u}(0,\cdot)\|_\infty \}.
\end{cases}
\end{equation*}
Then
$$u^*(t)\geq\bar{V}_\epsilon(t) \ \ for \ all \ t\geq T $$
and Lemma \ref{ode-lm} tells us that
$$\lim_{t\to\infty}u^*(t)-\bar{V}_\epsilon(t)=0.$$
Now we have
$$u^*(T)\geq\tilde{u}(0,x),  u^*(t+T)\geq\bar{V}_\epsilon(t+T)\geq\tilde{u}(t,0),
  u^*(t+T)\geq0=\tilde{u}(t,\tilde{h}(t)) \ \ for \ t\ge0.
$$
Hence, we can apply the comparison principle to deduce
$$\tilde{u}(t,x)\leq u^*(t+T) \ \ for \ t\geq0, 0<x<\tilde{h}(t).$$
As a consequence, there exists $\bar{T}>T$ such that
$$\tilde{u}(t,x)\leq(1-\epsilon)^{-1}\bar{V}_\epsilon(t+T)
\ \ for \ t\geq\bar{T}, 0\leq x \leq\tilde{h}(t).$$

From the Step 1, we know that there exists $L>l^*$ such that
$$v_\epsilon(t,x)>(1-\epsilon)\bar{V}_\epsilon(t)
\ \ for \ t>0, x\geq L.$$

We now define
$$\xi(t)=(1-\epsilon)^{-2}\int^t_0\mu(v_\epsilon)_x(s,0)ds+L+\tilde{h}(\bar{T})
\ \ for \ t\geq0,$$
$$w(t,x)=(1-\epsilon)^{-2}v_\epsilon(t,\xi(t)-x)
\ \ for \ t\geq0, 0\leq x \leq\xi(t).$$
Then
$$\xi^{'}(t)=(1-\epsilon)^{-2}\mu(v_\epsilon)_x(t,0),$$
$$-\mu w_x(t,\xi(t))=(1-\epsilon)^{-2}\mu(v_\epsilon)_x(t,0)$$
and so we have
$$\xi^{'}(t)=-\mu w_x(t,\xi(t)).$$
Clearly,
$$w(t,\xi(t))=0, \ \xi(T+\bar{T})\geq L+\tilde{h}(\bar{T}).$$
Moreover, for $0<x\leq\tilde{h}(\bar{T})$,
$$w(T+\bar{T},x)=(1-\epsilon)^{-2}v_\epsilon(T+\bar{T},\xi(T+\bar{T})-x)
\geq(1-\epsilon)^{-2}v_\epsilon(T+\bar{T},L)>
(1-\epsilon)^{-1}\bar{V}_\epsilon(T+\bar{T})\geq\tilde{u}(\bar{T},x)$$
and for $\tilde{h}(\bar{T})<x<\xi(0)$, $w(T+\bar{T},x)>0$.

And for $t\geq\bar{T}$, we have
$$w(t+T,0)=(1-\epsilon)^{-2}v_\epsilon(t+T,\xi(t+T))
\geq(1-\epsilon)^{-2}v_\epsilon(t+T,L)>
(1-\epsilon)^{-1}\bar{V}_\epsilon(t+T)\geq\tilde{u}(t,0).$$

Direct calculations show that, for $t\geq\bar{T}$ and
$0<x<\xi(t)$, with $\rho=\xi(t)-x$,
\begin{align*}
w_t-w_{xx}&=(1-\epsilon)^{-2}[(v_\epsilon)_t+(v_\epsilon)_\rho
\cdot\xi^{'}(t)-(v_\epsilon)_{\rho\rho}] \\
&=(1-\epsilon)^{-2}[\mu(1-\epsilon)^{-2}(v_\epsilon)_\rho(t,0)(v_\epsilon)_\rho(t,\rho)
+(v_\epsilon)_t-(v_\epsilon)_{\rho\rho}] \\
&\geq(1-\epsilon)^{-2}[\mu(v_\epsilon)_\rho(t,0)(v_\epsilon)_\rho(t,\rho)
+(v_\epsilon)_t-(v_\epsilon)_{\rho\rho}] \\
&=(1-\epsilon)^{-2}v_\epsilon(f(t,v_\epsilon)+\epsilon) \\
&\geq w(f(t,w)+\epsilon).
\end{align*}
Hence we can use Lemma \ref{com-principle free boun 2} to conclude that
$$w(t+T,x)\geq\tilde{u}(t,x) \ \ for \ t\geq\bar{T},
0<x<\tilde{h}(t)$$
$$\xi(t+T)\geq\tilde{h}(t) \ \ for \ t\geq\bar{T}.$$

It follows that
\begin{align*}
\overline{\lim}_{t\to\infty}\frac{h(t)}{t}=
\overline{\lim}_{t\to\infty}\frac{\tilde{h}(t-T)+l^*}{t}&\leq
\overline{\lim}_{t\to\infty}\frac{\xi(t)}{t} \\
&=\overline{\lim}_{t\to\infty}
\frac{(1-\epsilon)^{-2}\int^t_0\mu(v_\epsilon)_x(s,0)ds+L+\tilde{h}(\bar{T})}{t} \\
&=(1-\epsilon)^{-2}\lim_{t\to\infty}\frac{\int^t_0\mu(v_\epsilon)_x(s,0)ds}{t}.
\end{align*}

Note that 
$(v_\epsilon)_x(t,0)\to \tilde u^{**}_x(t,0)$ as $\epsilon\to 0$ uniformly in $t\in\RR$.
Thus,
\begin{equation}
\begin{split}
\label{6-7}
{\overline{\lim}}_{t\to\infty}\frac{h(t)}{t}
\leq\lim_{t\to\infty}\frac{\int^t_0\mu \tilde u^{**}_x(s,0)ds}{t}=c^*.
\end{split}
\end{equation}

\textbf{Step 3.}
We prove
$${\underline{\lim}}_{t\to\infty}\frac{h(t)}{t}\geq c^*.$$

By Lemma \ref{ode-lm} and Proposition \ref{main-results-of-part1}, we know that there
exists a unique positive almost periodic solution $v^*(t)$ of the problem
\begin{equation*}
v_t=vf(t+T,v)
\end{equation*}
and
\begin{equation}
\begin{split}
\label{6-9}
\lim_{t\to\infty}[\tilde{u}(t,x)-v^*(t)]=0
\end{split}
\end{equation}
locally uniformly in $x\ge0$. Using the comparison principle we have
\begin{equation}
\begin{split}
\label{6-10}
v^*(t) \geq\underline{V}_\epsilon(t+T).
\end{split}
\end{equation}
It then follows that $\underline{\lim}_{t\to\infty}[v^*(t)-\underline{V}_\epsilon(t+T)]\geq0$.

In view of $(\ref{6-9})$, we have
\begin{equation}
\begin{split}
\label{6-14}
\underline{\lim}_{t\to\infty}[\tilde{u}(t,x)-\underline{V}_\epsilon(t+T)]\geq0
\ \ locally \ uniformly \ in \ x\ge 0.
\end{split}
\end{equation}

By the same argument as Lemma \ref{semi-wave-lm3} we can get,
\begin{equation}
\begin{split}
\label{6-15}
z_\epsilon(t,x)\leq\underline{V}_\epsilon(t)
\ \ for \ 0<x<\infty.
\end{split}
\end{equation}

Due to the $(\ref{6-14})$ and $(\ref{6-15})$ we can find some $\tilde L>0$,
$\tilde{T}>T$, and define
$$\eta(t)=(1-\epsilon)^{2}\int^t_{\tilde{T}+T}\mu(z_\epsilon)_x(s,0)ds
+\tilde L
\ \ for \ t\geq\tilde{T}+T,$$
$$w(t,x)=(1-\epsilon)^{2}z_\epsilon(t,\eta(t)-x)
\ \ for \ t\geq\tilde{T}+T, 0\leq x \leq\eta(t)$$
such that
$$\tilde{u}(t,0)\geq w(t+T,0) \ \ for \ t\geq\tilde{T}$$
and
$$\tilde{u}(\tilde{T},x)\geq w(\tilde{T}+T,x) \ \ for \ 0\leq x \leq\eta(\tilde{T}+T).$$
Then
$$\eta^{'}(t)=(1-\epsilon)^{2}\mu(z_\epsilon)_x(t,0)$$
$$-\mu w_x(t,\eta(t))=(1-\epsilon)^{2}\mu(z_\epsilon)_x(t,0)$$
and so we have
$$\eta^{'}(t)=-\mu w_x(t,\eta(t)).$$
Clearly,
$$w(t,\eta(t))=0.$$

Direct calculations show that, for $t\geq\tilde{T}$ and
$0<x<\eta(t)$, with $\theta=\eta(t)-x$,
\begin{align*}
w_t-w_{xx}&=(1-\epsilon)^{2}[(z_\epsilon)_t+(z_\epsilon)_\theta
\cdot\eta^{'}(t)-(z_\epsilon)_{\theta\theta}] \\
&=(1-\epsilon)^{2}[\mu(1-\epsilon)^{2}(z_\epsilon)_\theta(t,0)(z_\epsilon)_\theta(t,\theta)
+(z_\epsilon)_t-(z_\epsilon)_{\theta\theta}] \\
&\leq(1-\epsilon)^{2}[\mu(z_\epsilon)_\theta(t,0)(z_\epsilon)_\theta(t,\theta)
+(z_\epsilon)_t-(z_\epsilon)_{\theta\theta}] \\
&=(1-\epsilon)^{2}z_\epsilon(f(t,z_\epsilon)-\epsilon) \\
&\leq w(f(t,w)-\epsilon).
\end{align*}
Hence we can use Lemma \ref{com-principle free boun 3} to conclude that
$$w(t+T,x)\leq\tilde{u}(t,x) \ \ for \ t\geq\tilde{T},
0<x<\eta(t+T),$$
$$\eta(t+T)\leq\tilde{h}(t) \ \ for \ t\geq\tilde{T}.$$

It follows that
\begin{align*}
\underline{\lim}_{t\to\infty}\frac{h(t)}{t}=
\underline{\lim}_{t\to\infty}\frac{\tilde{h}(t-T)+l^*}{t}&\geq
\underline{\lim}_{t\to\infty}\frac{\eta(t)}{t} \\
&=\underline{\lim}_{t\to\infty}
\frac{(1-\epsilon)^{2}\int^t_{\tilde{T}+T}\mu(z_\epsilon)_x(s,0)ds+\tilde{h}(\tilde{T})}{t} \\
&=(1-\epsilon)^{2}\lim_{t\to\infty}\frac{\int^t_0\mu(z_\epsilon)_x(s,0)ds}{t}.
\end{align*}

Note that 
$(z_\epsilon)_x(t,0)\to \tilde u^{**}_x(t,0)$ as $\epsilon\to 0$ uniformly in $t\in\RR$. Thus,
\begin{equation}
\begin{split}
\label{6-16}
{\underline{\lim}}_{t\to\infty}\frac{h(t)}{t}
\geq\lim_{t\to\infty}\frac{\int^t_0\mu \tilde u^{**}_x(s,0)ds}{t}=c^*.
\end{split}
\end{equation}

Hence, from $(\ref{6-7})$ and $(\ref{6-16})$ we have
$$\lim_{t\to\infty}\frac{h(t)}{t}=c^*.$$

{\bf Step 4.} We prove that  for any $\epsilon>0$,
$$\lim_{t\to\infty}\max_{x\leq(c^*-\epsilon)t}|u(t,x)-V^*(t)|=0.$$

 By the estimates for $\tilde{u}(t,x)$ given in Step 2 of the proof,
and for any given small
$\delta>0$, there exist $T^\delta>T$ and $R^\delta>0$ such that
$$u(t,x+l^*)\leq (1-\delta)^{-2}v_\delta(t,\xi(t)-x)
\ \ for \ t\geq T^\delta, 0\leq x\leq\tilde{h}(t)$$
where
$$\xi(t)=(1-\delta)^{-2}\int^t_0 \mu(v_\delta)_x(s,0)ds+R^\delta$$
and $v_\delta$ is the unique almost periodic solution of
\eqref{semi-wave-eq4} with $\epsilon$ replaced by $\delta$
and $\tilde{h}(t)=h(t+T)-l^*$.

Similarly, by Step 3 of the proof,
there exist $\tilde{T}^\delta, \bar{T}^\delta>T$ and $\tilde{R}^\delta$ such that
$$u(t,x+l^*)\geq(1-\delta)^2 z_\delta(t,\eta(t)-x)
\ \ for \ t\geq\tilde{T}^\delta, 0\leq x\leq\eta(t).$$
where
$$\eta(t)=(1-\delta)^2\int^t_{\bar{T}^\delta}\mu (z_\delta)_x(s,0)ds+\tilde{R}^\delta$$
and $z_\delta$ is the unique almost periodic solution of
\eqref{semi-wave-eq5} with $\epsilon$ replaced by $\delta$.

Since
$$\lim_{\delta\to0}(1-\delta)^{-2}\mu(v_\delta)_x(t,0)=
\lim_{\delta\to0}(1-\delta)^2\mu(z_\delta)_x(t,0)=\mu \tilde u^{**}_x(t,0)$$
uniformly for $t\geq0$, for any $\epsilon>0$, we can find
$\delta_\epsilon\in(0,\epsilon)$ small enough and
$T_\epsilon>0$ such that for all $t\geq T_\epsilon$, we have
$$|(1-\delta_\epsilon)^{-2}\int^t_0\mu(v_{\delta_\epsilon})_x(s,0)ds
-c^*t|<\frac{\epsilon}{2}t$$
and
$$|(1-\delta_\epsilon)^2\int^t_0\mu(z_{\delta_\epsilon})_x(s,0)ds
-c^*t|<\frac{\epsilon}{2}t$$

Let $\bar{R}^\delta=\tilde{R}^\delta-(1-\delta)^2\int^{\bar{T}^\delta}_0
\mu(z_\delta)_x(s,0)ds$. Choose $\bar{T_\epsilon}>T_\epsilon$ such that
$\bar{R}^\delta+\frac{\epsilon}{2}t>0$ for $t\geq\bar{T}_\epsilon$.
We now fix $\delta=\delta_\epsilon$ in $v_\delta$, $z_\delta$,
$\xi$ and $\eta$. Obviously, for $t\geq \bar{T}_\epsilon$,
$$\xi(t)-x\geq(c^*-\epsilon)t-x+R^{\delta_\epsilon}+\frac{\epsilon}{2}t$$
$$\eta(t)-x\geq(c^*-\epsilon)t-x+\bar{R}^{\delta_\epsilon}+\frac{\epsilon}{2}t$$
By Step 1, we have
$$\lim_{x\to\infty}z_{\delta_\epsilon}(t,x)=\underline{V}_{\delta_\epsilon}(t)
\  \ uniformly \ for \ t\in\RR$$
where $\underline{V}_{\delta_\epsilon}(t)$ is the unique positive
almost periodic solution of
$$(\underline{V}_{\delta_\epsilon})_t=\underline{V}_{\delta_\epsilon}
(f(t,\underline{V}_{\delta_\epsilon})-\delta_\epsilon)$$
and
$$\lim_{x\to\infty}v_{\delta_\epsilon}(t,x)=\bar{V}_{\delta_\epsilon}(t)
\  \ uniformly \ for \ t\in\RR$$
where $\bar{V}_{\delta_\epsilon}(t)$ is the unique positive almost
periodic solution of
$$(\bar{V}_{\delta_\epsilon})_t=\bar{V}_{\delta_\epsilon}
(f(t,\bar{V}_{\delta_\epsilon})+\delta_\epsilon)$$
Furthermore, by the same argument as Lemma \ref{semi-wave-lm3} we can find
$R^\epsilon>0$ such that for $x\geq R^\epsilon$,
$$v_{\delta_\epsilon}(t,x)\leq\bar{V}_{\delta_\epsilon}(t)
\ \ for \ all \ t\in\RR$$
and
$$z_{\delta_\epsilon}(t,x)\geq\underline{V}_{\delta_\epsilon}(t)-\epsilon
\ \ for \ all \ t\in\RR$$

It follows that, if
$$0\leq x\leq(c^*-\epsilon)t \ \ and \ \ t\geq
\max\{\frac{2(R^\epsilon-\bar{R}^{\delta_\epsilon})}{\epsilon},\bar{T}_\epsilon,
T^{\delta_\epsilon},\tilde{T}^{\delta_\epsilon}\}$$
then
$$u(t,x+l^*)\leq(1-\delta_\epsilon)^{-2}
v_{\delta_\epsilon}(t,\xi(t)-x)\leq(1-\delta_\epsilon)^{-2}
\bar{V}_{\delta_\epsilon}(t)$$
and
$$u(t,x+l^*)\geq(1-\delta_\epsilon)^{2}
z_{\delta_\epsilon}(t,\eta(t)-x)\geq(1-\delta_\epsilon)^{2}
[\underline{V}_{\delta_\epsilon}(t)-\epsilon]$$

So we take $T^*=\max\{\frac{2(R^\epsilon-\bar{R}^{\delta_\epsilon})}{\epsilon},\bar{T}_\epsilon,
T^{\delta_\epsilon},\tilde{T}^{\delta_\epsilon}\}$.
If $t\geq T^*$ and $l^*\leq x \leq(c^*-\epsilon)t$,
we have
$$(1-\delta_\epsilon)^{2}
[\underline{V}_{\delta_\epsilon}(t)-\epsilon]\leq u(t,x)
\leq(1-\delta_\epsilon)^{-2}
\bar{V}_{\delta_\epsilon}(t)$$
In the view of Step 1, this implies that
$$(1-\delta_\epsilon)^{2}
[\underline{V}_{\delta_\epsilon}(t)-\epsilon]
-\bar{V}_{\delta_\epsilon}(t)\leq u(t,x)-V^*(t)
\leq(1-\delta_\epsilon)^{-2}
\bar{V}_{\delta_\epsilon}(t)-\underline{V}_{\delta_\epsilon}(t)$$
Let
$$I(\epsilon)=\max\{|(1-\delta_\epsilon)^{2}
[\underline{V}_{\delta_\epsilon}(t)-\epsilon]
-\bar{V}_{\delta_\epsilon}(t)|, |(1-\delta_\epsilon)^{-2}
\bar{V}_{\delta_\epsilon}(t)-\underline{V}_{\delta_\epsilon}(t)|\}$$
Thus,
$$|u(t,x)-V^*(t)|\leq I(\epsilon).$$

By \eqref{6-6},
$$\lim_{t\to\infty}u(t,x)-V^*(t)=0 \ \ uniformly \ for \ x\in[0,l^*]$$
Hence we can find $\hat{T}>T^*$ such that
$$|u(t,x)-V^*(t)|\leq I(\epsilon) \ \ for \ t\geq\hat{T} \ and \ 0\leq x \leq l^*$$

Finally, we obtain for all $t\geq\hat{T}$ and $0\leq x \leq (c^*-\epsilon)t$,
$$|u(t,x)-V^*(t)|\leq I(\epsilon)$$
Let $\epsilon\to0$, we have $I(\epsilon)\to0$. So we get
$$\lim_{t\to0}\max_{x\leq(c^*-\epsilon)t}|u(t,x)-V^*(t)|=0.$$
The proof is now complete.

\end{proof}

\section{Remarks}

We have proved the existence of a unique spreading speed $c^*$ of \eqref{main-eq}
and the existence of a unique time almost periodic positive semi-wave solution of \eqref{aux-main-eq2}. It is seen that the spreading
speed of \eqref{main-eq} and the semi-wave solution of \eqref{aux-main-eq2} are closely related. In this section, we give some remarks
on the spreading speed of  the
 double fronts free boundary problem \eqref{main-doub-eq}.

First of all, note that the existence and uniqueness results for solutions of \eqref{main-eq} with given initial data $(u_0,h_0)$
 can be
extended to \eqref{main-doub-eq} using the same arguments as in Section 5
\cite{DuLi}, except that we need to modify the transformation in the
proof of Theorem 2.1 in \cite{DuLi} such that both boundaries are
straightened. In particular, for given
$g_0<h_0$ and $u_0$ satisfying
\begin{equation}
\begin{cases}
\label{initia-valu-doub}
u_0\in C^2([g_0,h_0],\RR^+) \cr
u_0(g_0)=u_0(h_0)=0 \ \ and  \ \ u_0>0 \ \ in \ \ (g_0,h_0),
\end{cases}
\end{equation}
the system \eqref{main-doub-eq} has a unique global solution $(u(t,x;u_0,h_0,g_0),h(t;u_0,h_0,g_0),g(t;u_0,h_0,g_0))$ with
$u(0,x;u_0,h_0,g_0)=u_0(x)$, $h(0;u_0,h_0,g_0)=h_0$, $g(0;u_0,h_0,g_0)=g_0$.
Moreover, $g(t)$ decreases and $h(t)$ increases as $t$ increases. Let
$$g_\infty=\lim_{t\to\infty}g(t;u_0,h_0,g_0)
\quad {\rm and}\quad h_\infty=\lim_{t\to\infty}h(t;u_0,h_0,g_0).
$$

We next note that, by (H3),  there is $L^*\ge 0$ such that
 $\inf_{ l\ge L^*} \tilde\lambda(a(\cdot),l)>0$, where $a(t)=f(t,0)$. By (H3) again, there is
 a unique time almost periodic and space homogeneous positive solution
 $V^*(t)$ of
\begin{equation}
\label{glob-eq}
u_t=u_{xx}+uf(t,u) \quad x\in(-\infty,\infty).
\end{equation}
Moreover,
 for any $u_0\in C^b_{\rm unif}(\RR,\RR^+)$ with
$\inf_{x\in(-\infty,\infty)}u_0(x)>0$,
$\lim_{t\to\infty}\|u(t,\cdot;u_0)-V^*(t)\|_{\infty}=0$.
The following proposition then follows from \cite[Proposition 6.2]{LiLiSh1}.

\begin{proposition}
\label{spreading-vanishing-doub-prop}
Assume (H1)-(H3). Let $u_0$ satisfying \eqref{initia-valu-doub} and $g_0<h_0$ be given.
\begin{itemize}
\item[(1)]
Either

(i) $h_\infty-g_\infty\le L^*$ and
$\lim_{t\to+\infty} \|u(t,\cdot;u_0,h_0,g_0)\|_{C([g(t),h(t)])}=0$ (i.e. vanishing occurs)

or

(ii)  $h_{\infty}=-g_{\infty}=\infty$ and
$\lim_{t\to\infty}[u(t,x;u_0,h_0,g_0)-V^*(t)]=0$
locally uniformly for $x\in(-\infty,\infty)$ (i.e. spreading occurs).

\item[(2)] If $h_0-g_0\ge L^*$, then $h_\infty=-g_{\infty}=\infty$.

\item[(3)]
Suppose $h_0-g_0<L^*$. Then there exists $\mu^{*}>0$
such that spreading occurs if $\mu>\mu^{*}$ and vanishing occurs if $\mu\le \mu^{*}$.
\end{itemize}
\end{proposition}

We now have

\begin{proposition}
\label{spreading-speed-doub-prop}
For given $g_0<h_0$ and  $u_0$ satisfying \eqref{initia-valu-doub}, if spreading occurs,
then
$$
\lim_{t\to\infty}\frac{h(t)}{t}=\lim_{t\to\infty}\frac{-g(t)}{t}=c^*,
$$
where $c^*$ is the spreading speed of \eqref{main-eq}.
\end{proposition}

\begin{proof}
Let $g_0<h_0$ and  $u_0$ satisfying \eqref{initia-valu-doub} be given. Assume that
$g_\infty=-\infty$ and $h_\infty=\infty$. Then there are $T^*>0$ and $N^*>0$ such that
$$
-N^*L^*<g(T^*)<-L^*<L^*<h(T^*)<N^*L^*.
$$
Without loss of generality, we may assume that $g_0<-L^*<L^*<h_0$ and $u_0(x)>0$ for $g_0<x<h_0$.
Note that there are $u_0^- \in C([-L^*,L^*],\RR^+)$ and $u_0^+\in C([-N^*L^*,N^*L^*],\RR^+)$
such that
$$
\begin{cases}
u_0^-(-x)=u_0^-(x)\,\, {\rm for}\,\, -L^*<x<L^*\cr
 u_0^-(\pm L^*)=0\cr
  u_0^-(x)<u_0(x)\,\,{\rm for}\,\, -L^*<x<L^*,
  \end{cases}
$$
and
$$
\begin{cases}
u_0^+(-x)=u_0^+(x)\,\, {\rm for}\,\, -N^*L^*<x<N^*L^*\cr
 u_0^+(\pm N^* L^*)=0\cr
   u_0(x)<u_0^+(x)\,\,{\rm for}\,\, -N^*L^*<x<N^*L^*.
\end{cases}
$$
Hence
$$
\begin{cases}
u(t,x;u_0^-,L^*,-L^*)\le u(t,x;u_0,h_0,g_0)\quad {\rm for}\quad g(t;u_0^-,L^*,-L^*)<x<h(t;u_0^-,L^*,-L^*)\cr
u(t;x,u_0,g_0,h_0)\le u(t,x;u_0^+,N^*L^*,-N^*L^*)\quad {\rm for}\quad g(t;u_0,h_0,g_0)<x<h(t;u_0,h_0,g_0).
\end{cases}
$$
Note that
$$
\begin{cases}
u(t,-x;u_0^-,L^*,-L^*)=u(t,x;u_0^-,L^*,-L^*)\cr
h(t;u_0^-,L^*,-L^*)=-g(t;u_0^-,L^*,-L^*)
\end{cases}
$$
and
$$
\begin{cases}
u(t,-x;u_0^+,L^*,-L^*)=u(t,x;u_0^+,L^*,-L^*)\cr
h(t;u_0^+,N^*L^*,-N^*L^*)=-g(t;u_0^+,N^*L^*,-N^*L^*).
\end{cases}
$$
Then $u_x(t,0;u_0^-,L^*,-L^*)=u_x(t,0;u_0^+,N^*L^*,-N^*L^*)=0$. This together with Theorem \ref{main-thm2}
implies that
\begin{equation*}
c^*=\lim_{t\to\infty}\frac{h(t;u_0^-,L^*,-L^*)}{t}\le \lim_{t\to\infty}\frac{h(t;u_0,h_0,g_0)}{t}\le \lim_{t\to\infty}\frac{h(t;u_0^+,N^*L^*,-N^*L^*)}{t}=c^*
\end{equation*}
and
$$
c^*=\lim_{t\to\infty}\frac{-g(t;u_0^-,L^*,-L^*)}{t}\ge \lim_{t\to\infty}\frac{-g(t;u_0,h_0,g_0)}{t}\ge \lim_{t\to\infty}\frac{-g(t;u_0^+,N^*L^*,-N^*L^*)}{t}=c^*
$$
Hence
$$
\lim_{t\to\infty}\frac{h(t;u_0,h_0,g_0)}{t}=\lim_{t\to\infty}\frac{-g(t;u_0,h_0,g_0)}{t}=c^*.
$$
\end{proof}

\section*{Acknowledgements}

Fang Li would like to thank the China Scholarship Council for
financial support during the two years of her overseas study and to
express her gratitude to the  Department of Mathematics and
Statistics, Auburn University  for its kind hospitality.


\begin{thebibliography}{99}

\bibitem{ASB} Angenent, S.B.: The zero set of a solution of a parabolic equation, {\it J. Reine Angew. Math}
 {\bf 390}  (1988), 79-96.


\bibitem{BeHa} H. Berestycki and F. Hamel,
 Generalized transition waves and their properties,
 {\it Comm. Pure Appl. Math.} {\bf  65} (2012), no. 5, 592-648.



\bibitem{BeHaNa} H. Berestycki, F. Hamel, and G. Nadin,
 Asymptotic spreading in heterogeneous diffusive excitable media,
 {\it  J. Funct. Anal.} {\bf 255} (2008), no. 9, 2146-2189.

 \bibitem{BeHaNadir} H. Berestycki, F. Hamel, and N. Nadirashvili,
 The speed of propagation for KPP type problems. I: Periodic framework, {\it J. Eur. Math. Soc.} {\bf 7} (2005), 173-213.

\bibitem{BeNa} H. Berestycki and G.  Nadin,
 Spreading speeds for one-dimensional monostable reaction-diffusion equations,
 {\it J. Math. Phys.} {\bf 53} (2012), no. 11, 115619, 23 pp.




\bibitem{DuGu} Y. Du and Z. Guo, Spreading-vanishing dichotomy in a diffusive logistic model with a free boundary, II.,
{\it J. Differential Equations}, {\bf 250} (2011), 4336-4366.

\bibitem{DuGuPe} Y. Du, Z. Guo, and R.Peng, A diffusive logistic model with a free boundary in time-periodic environment,
{\it J. Functional Analysis}, {\bf 265} (2013), 2089-2142.



\bibitem{DuLia} Y. Du and X. Liang, Pulsating semi-waves in periodic media and spreading speed determined by a free boundary model, {\it Ann. I. H. Poincar\'{e}-AN},
{\bf 32} (2013), 279-305.



\bibitem{DuLi} Y. Du and Z. Lin, Spreading-vanishing dichotomy in the diffusive logistic model with a free
boundary, {\it SIAM J. Math. Anal.}, {\bf 42} (2010), 377-405.





\bibitem{DuLou} Y. Du and B. Lou, Spreading and vanishing in nonlinear diffusion problems with free boundaries,
{\it J. Eur. Math. Soc.}, in press. arXiv:1301.5373

\bibitem{Fink} A.M. Fink, Almost Periodic Differential Equations, Lectures Notes in Mathematics, Vol. 377, Springer-Verlag, Berlin-New York, 1974.


\bibitem{Fis} R. Fisher, The wave of advance of advantageous genes,
{\it Ann. of Eugenics}, {\bf 7}(1937), pp. 335-369.

\bibitem{HuSh} J. Huang and W. Shen,
Speeds of spread and propagation of KPP models in time almost and space periodic media,
{\it  SIAM J. Appl. Dyn. Syst.} {\bf 8} (2009), no. 3, 790-821.


\bibitem{KPP} A. Kolmogorov, I. Petrowsky, and N.Piscunov,
A study of the equation of diffusion with increase in the quantity of matter, and its application to a
biological problem.
 {\it Bjul. Moskovskogo Gos. Univ.}, {\bf 1} (1937), pp. 1-26.

\bibitem{KoSh} L. Kong and W. Shen,  Liouville type property and spreading speeds of KPP equations in periodic media with localized spatial inhomogeneity,
{\it  J. Dynam. Differential Equations} {\bf 26} (2014), no. 1, 181-215.



\bibitem{LiLiSh1} F. Li, X. Liang, and W. Shen, Diffusive KPP Equations with Free Boundaries in Time Almost Periodic Environments: I. Spreading and Vanishing
Dichotomy, submitted.

\bibitem{LiZh} X. Liang and X.-Q. Zhao,
 Spreading speeds and traveling waves for abstract monostable evolution systems,
 {\it  J. Funct. Anal.} {\bf 259} (2010), no. 4, 857-903.

\bibitem{LoHoMa}  J. L. Lockwood, M. F. Hoppes, M. P. Marchetti, Invasion Ecology, Blackwell Publishing,
2007.

\bibitem{MiSh1} J. Mierczynski and W. Shen,
Exponential separation and principal Lyapunov exponent/spectrum for
random/nonautonomous parabolic equations, {\it J. Differential
Equations},  {\bf 191}, (2003), 175-205.




\bibitem{MiSh3} J. Mierczynski and W. Shen,
``Spectral Theory for Random and Nonautonomous Parabolic Equations
and Applications,''  Chapman \& Hall/CRC Monogr. Surv. Pure Appl.
Math., Chapman \& Hall/CRC, Boca Raton, FL, 2008.




\bibitem{Nad}  G.  Nadin,
Traveling fronts in space-time periodic media, {\it J. Math. Pures Appl.}  {\bf (9) 92} (2009), no. 3, 232-262.

\bibitem{NaRo} G. Nadin and L.  Rossi,
 Propagation phenomena for time heterogeneous KPP reaction-diffusion equations,
 {\it  J. Math. Pures Appl.} {\bf (9) 98} (2012), no. 6, 633-653.


 \bibitem{NoRoRyZl} J.   Nolen, J.-M. Roquejoffre, L. Ryzhik, and A. Zlato${\rm \check{s}}$,
  Existence and non-existence of Fisher-KPP transition fronts, {\it  Arch. Ration. Mech. Anal.} {\bf 203} (2012), no. 1, 217-246.

  \bibitem{NoXi} J. Nolen and J. Xin,  Existence of KPP type fronts in space-time periodic shear flows and a study of minimal speeds based on variational principle, {\it  Discrete Contin. Dyn. Syst.} {\bf 13} (2005), no. 5, 1217-1234.

 \bibitem{NoXi1} J.     Nolen and J. Xin,
  A variational principle based study of KPP minimal front speeds in random shears, {\it Nonlinearity} {\bf 18} (2005), no. 4, 1655-1675.

 \bibitem{PA} A.Pazy, Semigroups of linear operators and application to
partial differential equations, {\it Springer-Verlag, New York}, (1983).

\bibitem{She1} W. Shen,
 Variational principle for spreading speeds and generalized propagating speeds in time almost periodic and space periodic KPP models,
  {\it Trans. Amer. Math. Soc.} {\bf  362} (2010), no. 10, 5125-5168.

\bibitem{She2} W. Shen,  Existence, uniqueness, and stability of generalized traveling waves in time dependent monostable equations,
 {\it J. Dynam. Differential Equations} {\bf  23} (2011), no. 1, 1-44.

 \bibitem{She3} W. Shen, Existence of generalized traveling waves in time recurrent and space periodic monostable equations,
  {\it J. Appl. Anal. Comput.} {\bf 1}  (2011), no. 1, 69-93.

\bibitem{ShYi0} W. Shen and Y.  Yi, Yingfei Convergence in almost periodic Fisher and Kolmogorov models,
{\it J. Math. Biol.} {\bf 37} (1998), no. 1, 84-102.

\bibitem{ShYi} W. Shen and Y. Yi, Almost automorphic and almost periodic
dynamics in skew-product semiflow, {\it Memoirs of the American Mathmatical
Society},  {\bf 647}, (1998).

\bibitem{ShKa} N. Shigesada and K. Kawasaki, ``Biological Invasions: Theory and Practice," Oxford Series
in Ecology and Evolution, Oxford, Oxford University Press, 1997.

\bibitem{TaZhZl} T. Tao, B. Zhu, annd A.  Zlato${\rm \check{s}}$,  Transition fronts for inhomogeneous monostable reaction-diffusion equations via linearization at zero, {\it  Nonlinearity} {\bf 27}  (2014), no. 9, 2409-2416.

\bibitem{Wei}    H. Weinberger, On spreading speed and travelling waves for growth and migration models in a periodic habitat,
{\it J. Math. Biol.} {\bf 45} (2002)
511-548.

\bibitem{MLZ} M. Zhou, The asymptotic behavior of the Fisher-KPP
equation with free boundary, preprint.


\bibitem{Zla} A. Zlato${\rm \check{s}}$,  Transition fronts in inhomogeneous Fisher-KPP reaction-diffusion equations,
 {\it J. Math. Pures Appl.} {\bf (9) 98} (2012), no. 1, 89-102.




\end{thebibliography}
\end{document}